\documentclass[a4paper,10pt]{amsart}
\usepackage[latin1]{inputenc}

\usepackage{multicol}
\usepackage{amssymb, amsmath, amsthm}
\usepackage{latexsym}
\usepackage{dsfont}
\usepackage{graphicx}

\usepackage{curves} 
\usepackage{epic} 
\usepackage{eepic}
\usepackage{xcolor}
 
\input xy
\xyoption{all}

\usepackage{verbatim}

\usepackage{hyperref}

\newtheorem{theorem}{Theorem}[section]

\newtheorem{corollary}[theorem]{Corollary}

\newtheorem{lemma}[theorem]{Lemma}

\newtheorem{proposition}[theorem]{Proposition}

	{\par\noindent{\bf Proposition \ref{res:hiper}.}\!\!
	\nopagebreak\normalsize\it}%

	{\par\noindent{\bf Theorem \ref{result43}.}\!\!
	\nopagebreak\normalsize\it}%

	{\par\noindent{\it Sketch of the proof}.  
	\nopagebreak\normalsize}%
	{\hfill\linebreak[2]\hspace*{\fill}$\circlearrowleft$}

	{\par\noindent{\it Proof of Proposition }\ref{prop:stab:smc}.  
	\nopagebreak\normalsize}%
	{\hfill\linebreak[2]\hspace*{\fill}$\circlearrowleft$}

\newenvironment{proof2}%
	{\par\noindent{\it Proofs of Propositions }\ref{adap:mon}{\it and }\ref{simult:adap}.\!\!\!
	\nopagebreak\normalsize}%
	{\hfill\linebreak[2]\hspace*{\fill}$\circlearrowleft$}

{\theoremstyle{definition}
       \newtheorem{definition}[theorem]{Definition}

       \newtheorem{remark}[theorem]{Remark}
       
       \newtheorem{parrafo}[theorem]{{\!}}  }

\numberwithin{equation}{theorem}

\newcommand{\nat}{\mathbb N}
 
\newcommand{\calo}{{\mathcal {O}}}

\DeclareMathOperator{\ord}{ord}

\DeclareMathOperator{\Sing}{Sing}
\DeclareMathOperator{\Spec}{Spec}

\DeclareMathOperator{\In}{In}
\DeclareMathOperator{\Gr}{Gr}

\newcommand{\R}{{\mathcal R}}
\newcommand{\U}{{\mathcal U}}
\newcommand{\G}{{\mathcal G}}
\newcommand{\C}{{\mathcal C}}

\renewcommand{\L}{{\mathcal L}}

\newcommand{\p}{{\mathfrak p}}
\renewcommand{\S}{{\mathbb S}}
\newcommand{\Z}{{\mathbb Z}}

\newcommand{\m}{{\mathcal{M}}}
\renewcommand{\P}{\mathcal{P}}

\newcommand{\A}{\mathbb{A}}

\newcommand{\id}[1]{\langle #1 \rangle}
\newcommand{\x}{{\mathbf{x}}}

\definecolor{darkpurple}{rgb}{0.28,0.24,0.55}
\definecolor{lightblue}{rgb}{0,0.75,1}

\title{Monoidal transforms and invariants of singularities in positive characteristic}

\author{Ang\'elica Benito}
\author{Orlando E. Villamayor U.}
 \thanks{The authors are partially supported by MTM2009-07291.}
\date{\today}

\address{Dpto. Matem\'aticas,  Universidad
Aut\'onoma de Madrid and ICMAT-UAM\\
Ciudad Universitaria de Cantoblanco, 28049 Madrid, Spain}

\email[Ang\'elica Benito]{angelica.benito@uam.es}
\email[Orlando E. Villamayor U.]{villamayor@uam.es}


\begin{document}

\maketitle


\begin{abstract}
The problem of resolution of singularities in positive characteristic can be reformulated as follows: Fix a hypersurface $X$, embedded in a smooth scheme, with points of multiplicity at most $n$. Let an $n$-sequence of transformations of $X$ be a finite composition of monoidal transformations with centers included in the $n$-fold points of $X$, and of its successive strict transforms. The open problem (in positive characteristic) is to prove that there is an $n$-sequence such that the final strict transform of $X$ has no points of multiplicity $n$ (no $n$-fold points).

In characteristic zero, such an $n$-sequence is defined in two steps: the first consisting in the transformation of $X$ to a hypersurface with $n$-fold points in the so called \emph{monomial case}. The second step consists in the elimination of these $n$-fold points (in the monomial case), which is achieved by a simple combinatorial procedure for choices of centers.

The invariants treated in this work allow us to define a notion of \emph{strong monomial case} which parallels that of monomial case in characteristic zero: If a hypersurface is within the strong monomial case we prove that a resolution can be achieved in a combinatorial manner.
\end{abstract}

%
%
%
%
%
%
%
%
%
%
%
%
%
%
%

{\tableofcontents}

\section{{Introduction.}}\label{intro}
\begin{parrafo}
The objective of this paper is to study invariants of singularities in positive characteristic. A particular motivation is to give invariants that would yield a sequence of monoidal transformations to eliminate the points of highest multiplicity of a hypersurface $X$. To be precise, let $V$  be a smooth scheme of dimension $d$ over a perfect field $k$ of characteristic $p>0$, and let $X$ be a hypersurface in $V$ with highest multiplicity $n$. The problem is to define a sequence 
\begin{equation}\label{opq}
\xymatrix@R=0pc@C=0pc{
X & & & & & X_1 &  & & & &  &  & & & &   X_r  & & & & & & & & & & X_N\\
V  &  & & & &   V_{1}\ar[lllll]_{\pi_{C_0}}   & & & & & \dots \ar[lllll]_{\pi_{C_1}} &  & & & &   V_{r}\ar[lllll]_{\pi_{C_{r-1}}}   & & & & & \dots \ar[lllll]_{\pi_{C_{r}}} &  & & & &   V_{N}\ar[lllll]_{\pi_{C_{N-1}}}  
}\end{equation}
where each $V_{i-1}\overset{\pi_i}\longleftarrow V_{i}$ is a monoidal transformation with center $C_{i-1}$ included in the $n$-fold points of 
$X_{i-1}$, so that $X_N$ has no point of multiplicity $n$. Here each $X_i\subset V_{i}$ denotes the strict transform of $X_{i-1}$ by $\pi_{C_{i-1}}$. We require, in addition, that the exceptional locus of $V\longleftarrow V_{N}$ is a union of $N$ hypersurfaces with normal crossings at $V_{N}$.  A sequence with this property is said to define a \emph{simplification of the $n$-fold points} of $X$.

In characteristic zero, simplifications of $n$-fold points of $X$ are known to exist. This is usually done in two steps. The first step consists of a sequence of, say  $r$, monoidal transformations, so that the set of points of highest multiplicity $n$ of $X_r$ is within the so called \emph{monomial case}. The second step consists of the elimination of the $n$-fold points of the  hypersurface $X_r$, which is assumed to be in the monomial case. The latter step is rather simple, and it can be achieved by a  combinatorial choice of centers. 

Both steps rely on Hironaka's main \emph{inductive invariant}, say $\ord^{(d-1)}(x)\in\mathbb{Q}$, defined for $x$ in the highest multiplicity locus of the hypersurface. In fact, these invariants lead to the construction  of a sequence in a such a way that $X_r$ is in the monomial case. 
The role of Hironaka's main inductive function in both steps mentioned above, always in characteristic zero, will be recalled in \ref{pi13}.

In this work we study Hironaka's inductive function over perfect fields of arbitrary characteristic. We will  introduce the notion of \emph{strong monomial case} for a hypersurface in positive characteristic. This notion will be characterized in terms of Hironaka's inductive functions. It  parallels that of monomial case in characteristic zero, i.e.,  if $X_r$ is in the strong monomial case, then elimination of $n$-fold points is achieved in a combinatorial manner.

In the case of hypersurfaces in positive characteristic, a canonical sequence of transformations of $X$ was defined in \cite{BV3}. This sequence transforms  $X$  to an embedded hypersurface, say $X_r$,
which is closely related to the {\em monomial case}, but still weaker than the \emph{strong monomial case} treated here.
The simplification of $n$-fold problem  would be solved if one could fill the gap between the weak monomial case in \cite{BV3} and our strong monomial case. To be precise, the open problem of simplification (and of resolution of singularities) would be solved if one can define a sequence of monoidal transformations that transforms a hypersurface in the monomial case into one in the strong monomial case.

 This can be easily achieved in low dimension, and we prove resolution of singularities of  $2$-dimensional schemes by means of the invariants introduced here. A detailed proof of this fact can be found in \cite{BeVsup}. 
\end{parrafo}

\begin{parrafo}\label{pi12}
Assume, for simplicity, that $V$ is affine and $X=V(f)$ is a hypersurface with highest multiplicity $n$. We will first attach to the previous data the algebra $\calo_{V}[fW^n](\subset \calo_{V}[W]$) with $n$ as above. Namely, the $\calo_{V}$-subalgebra of $\calo_{V}[W]$ generated by the element $fW^n$. The notion of transformation of  hypersurfaces (with the center included in the subset of $n$-fold points) has a natural reformulation in the language of algebras. Moreover, the task of defining a sequence (\ref{opq}), that eliminates the $n$-fold points of $X$, by means of monoidal transformations, can be also expressed in terms of algebras and transformations of algebras (see \ref{par1}).

This reformulation of  the simplification problem in terms of algebras is well justified. In fact, the original algebra $\calo_{V}[fW^n]$ can be extended canonically to a so called \emph{differential algebra} so both are strongly linked: for the purpose of constructing a simplification of $\calo_{V}[fW^n]$, there is no harm in replacing it by its differential extension. Over fields of characteristic zero, this procedure is well-known.
In fact, differential algebras (see \ref{rp23}) are closely related to the theory of maximal contact in characteristic 0. In such context, hypersurfaces of maximal contact allow us to reformulate the problem of simplification with the simplification of a new algebra, defined over a smooth hypersurface $\overline{V}$, and hence in one dimension less. $\overline{V}$ is  called a \emph{hypersurface of maximal contact}. This form of induction is formulated in the language of algebras, the correspondient algebra defined over $\overline{V}$ is known as the \emph{coefficient algebra}.
\end{parrafo}

\begin{parrafo}\label{pi13}  In problems of resolution of singularities, it is natural to consider sequences of transformations of the form
\begin{equation}\label{opqA}
\xymatrix@R=0pc@C=0pc{
V  &  & & & &   V_{1}\ar[lllll]_{\pi_{C_0}}   & & & & & \dots \ar[lllll]_{\pi_{C_1}} &  & & & &   V_{r}\ar[lllll]_{\pi_{C_{r-1}}}
}\end{equation}
with the additional condition that the exceptional locus of the sequence, say $E_r=\{H_1,\dots,H_r\}$, are hypersurfaces having only normal crossings at $V_r$. A \emph{monomial algebra} in $V_r$ is an algebra of the form
$\calo_{V_r}[I(H_1)^{\alpha_1}\dots I(H_r)^{\alpha_r}W^s]$ for some $s,\alpha_i\in\mathbb{Z}_{\geq 0}$.

In the case of characteristic 0, the simplification of $n$-fold points can be achieved in two steps, both of them expressed in terms of algebras, once a hypersurface of maximal contact, say $\overline{V}$, is fixed:

\vspace{0.2cm}

\noindent{\sf{\bfseries (STEP 1)}} in which a sequence of monoidal transformations is defined over the hypersurface of maximal contact, say 
\begin{equation}\label{opq2}
\xymatrix@R=0pc@C=0pc{
\overline{V}  &  & & & &   \overline{V}_{1}\ar[lllll]_{\pi_0}   & & & & & \dots \ar[lllll]_{\pi_1} &  & & & &   \overline{V}_{r}\ar[lllll]_{\pi_{r-1}},
}\end{equation}
so that the coefficient algebra is transformed into a monomial algebra supported on the exceptional locus, say  $\calo_{\overline{V}_r}[I(H_1)^{\alpha_1}\dots I(H_r)^{\alpha_r}W^s]$. This sequence can be defined so as to induce a sequence (\ref{opqA}), and 
in this case, the $n$-fold points of $X_r$ (the strict transform of $X$) are said to be in the {\em monomial case}.

\vspace{0.2cm}

\noindent{\sf{\bfseries (STEP 2)}} in which a simplification of the $n$-fold points of $X_r$ (monomial case) is defined, say
\begin{equation}\label{opqB}
\xymatrix@R=0pc@C=0pc{
V_r  &  & & & &   V_{r+1}\ar[lllll]_{\pi_{r}}   & & & & & \dots \ar[lllll]_{} &  & & & &   V_{N}\ar[lllll]_{\pi_{N-1}}
}\end{equation}
 This step is achieved in an easy combinatorial manner. This procedure of choice of centers is defined only in terms of the exponents $\alpha_i$ of the monomial algebra obtained in Step 1.
 
 \vspace{0.2cm}

All these arguments (always in characteristic 0), rely strongly 
on the Hironaka's inductive functions $\ord^{(d-1)}$ (see (\ref{dford})), defined in terms of the coefficient algebra. In fact, Hironaka's functions allow us to attach to an arbitrary sequence (\ref{opqA}) a monomial algebra $\calo_{V_r}[I(H_1)^{\alpha_1}\dots I(H_r)^{\alpha_r}W^s]$. To be precise, this is done by setting $\frac{\alpha_i}{s}+1=\ord_{i-1}^{(d-1)}(y_{i-1})$ ($i=1,\dots,r$); here  the right hand side is the evaluation of the inductive function at $y_{i-1}$, the generic point of the center $C_{i-1}$.
\end{parrafo}

\begin{parrafo}{\bf Main objetives of this work}\label{pi14}. In this work we consider schemes over perfect fields of positive characteristic. The two main objectives are:
\begin{enumerate}
\item[(1)]  to define an analogue to Hironaka's inductive functions, called here $v-ord_i^{(d-1)}$ (Main Theorem 1 in \ref{result43}), with values in $\mathbb{Q}$. These functions enable us to attach a monomial algebra $\calo_{V_r}[I(H_1)^{h_1}\dots I(H_r)^{h_r}W^s]$ to a sequence of transformations (\ref{opqA}),  setting as before $\frac{h_i}{s}+1=v-ord_{i-1}^{(d-1)}(y_{i-1})$ (see Main Theorem 2 in \ref{MT2}). 

\item[(2)] To characterize, by numerical invariants, a case called here \emph{strong monomial case} (Definition \ref{def_smc}), in which a combinatorial resolution of the monomial algebra defines, as in Step 2, a simplification of $n$-fold points (Theorem \ref{claim:gamma:res}). This property will rely strongly on Main Theorem 2.
\end{enumerate}
\end{parrafo}

\begin{parrafo}{\bf Differences with characteristic zero}. In characteristic zero, Hironaka's inductive function $\ord^{(d-1)}$ is upper semi-continuous. This property follows from a form of coherence, and the proof of this property requires some form of patching of local data, and all together it is quite involved. In positive characteristic the function $v-ord^{(d-1)}$ is not upper semi-continuous and therefore we do not go through this kind of difficulty. So there is no coherence or patching to be proved in the positive characteristic case. Despite this fact, this function is essential in the study of singularities and we show that it leads to (1) and (2) in \ref{pi14}. 

In characteristic zero the value of the function $\ord^{(d-1)}$, at a given point, is computed by fixing a hypersurface of maximal contact.  
As there is no maximal contact in positive characteristic, we replace reduction to hypersurfaces of maximal contact by \emph{transversal projections}: $V^{(d)}\longrightarrow V^{(d-1)}$ defined in \'etale topology (Definition \ref{trpt}). In this setting, algebras over the smooth scheme $V^{(d-1)}$ are defined; they are called \emph{elimination algebra} (\ref{csc}). In characteristic zero elimination algebras parallel the role of the coefficients algebras.

We use here transversal projections and elimination algebras to compute the value of the function $v-ord^{(d-1)}$ at a given point, which is a rational number. To fix ideas let $x$ be an $n$-fold point of $X=V(f)\subset V^{(d)}$. The Weierstrass Preparation Theorem ensures that one can choose a regular system of parameters $\{z,x_1,\dots,x_{r-1}\}$ so that at the completion $\widehat{\calo}_{V^{(d)},x}=k'[[z,x_1,\dots,x_{r-1}]]$ ($r=d$ if $x$ is closed) we can take
\begin{equation}\label{eqintrof}
f(z)=z^n+a_1z^{n-1}+\dots+a_n\ \hbox{ with }\ a_i\in k'[[x_1,\dots,x_{r-1}]].
\end{equation}
A rational number $\geq1$ is defined as
\begin{equation}\label{eqmax}
\max_ z
\Big\{\min_{1\leq i\leq n}\Big\{\frac{\nu_x(a_i)}{i}\Big\}\Big\}\in\mathbb{Q},\ \ (\nu_x(a_i)\hbox{ is the order at }k'[[x_1,\dots,x_{r-1}]]).
\end{equation}
This is the maximal \emph{slope}, for the different choices of $z$, but always fixing the inclusion of rings  $k'[[x_1,\dots,x_{r-1}]]\subset k'[[z,x_1,\dots,x_{r-1}]]=\widehat{\calo}_{V^{(d)},x}$. Fixing an inclusion of rings is formulated here by fixing a morphism of smooth schemes $V^{(d)}\longrightarrow V^{(d-1)}$ (called here \emph{projection}). In order to parallel the presentation in  (\ref{eqintrof}) (Weierstrass Preparation Theorem) we need to consider \'etale topology. Projections for which $X$ can be expressed by an equation as in (\ref{eqintrof}) (where $n$ is the multiplicity of $X$ at the point), will be said to be \emph{transversal} at $x$.

Our setting will be slightly more general. Once a transversal projection $V^{(d)}\overset{\beta}{\longrightarrow} V^{(d-1)}$ is fixed, we will consider an expression
\begin{equation}\label{eqintrof2}
f(z)=z^n+a_1z^{n-1}+\dots+a_n\in\calo_{V^{(d-1)}}[z]
\end{equation}
where $a_i$ are global functions on $V^{(d-1)}$ and where $z$ is a global function on $V^{(d)}$ so that $\{dz\}$ is a basis of $\Omega_\beta^1$, the sheaf of $\beta$-relative differentials. In this cases,  the smooth hypersurface $\{z=0\}$ is a section of $V^{(d)}\overset{\beta}{\longrightarrow}V^{(d-1)}$. We will abuse the notation and say that the function $z$ is a {\em transversal section} of $\beta$. 

We show here that the rational number in (\ref{eqmax}) is independent of the chosen transversal projection, and hence intrinsic of the singularity (Main Theorem 1). This defines a rational invariant attached to singular point $x$, denoted here by $v-ord^{(d-1)}(x)$. 

If we fix two $n$-fold points $x$ and $y$, so that $x\in\overline{y}$, then it will be shown that $v-ord^{(d-1)}(x)\geq v-ord^{(d-1)}(y)$ (despite this property, the function is not upper semi-continuous). This inequality will be used in the proof of the two main objectives (1) and (2) in \ref{pi14}.

This invariant attached to the singularity has been largely studied in positive characteristic for the particular case of equations of the form $ f_{p^{e}}(z)=z^{p^{e}}+a_{p^{e}}\in\calo_{V^{(d-1)}}[z]$ (the purely inseparable case), e.g. \cite {Co} ,\cite{HaWa}, \cite{HaKang}, \cite{Moh96}. This equation involves a particular transversal projection, say $V^{(d)}\overset{\beta}{\longrightarrow}V^{(d-1)}$. Note that pure inseparability fails to hold if the projection is changed (pure inseparability is not a property of the singularities of a hypersurface). Our result shows that the rational number in (\ref{eqmax}), usually called the \emph{slope of the singularity}, is independent of the projection and hence intrinsic of the singularity.

%
%
%
\end{parrafo}

\begin{parrafo}{\large{\bf Organization and further comments}}.

\vspace{0.3cm}

\noindent {\sf Part I: $p$-presentations, adaptations and the tight monomial algebra}.

\vspace{0.2cm}

The objective of this first part is the definition of the inductive function and the study of its main properties mentioned in \ref{pi14}. This leads to the two main Theorems stated in the last Section \ref{secMT},  we suggest a first look at this last section for an overall view of the preliminary results that are needed. 

This first part is developed so as to introduce gradually the inductive function in positive characteristic, and to pave the way to the study of the strong monomial case in Part II. This part has been organized so as to present only those technical aspects which are crucial in the first two parts, whereas other technical arguments are gathered in Part III.

Section \ref{sec111} encompasses several notions used throughout the paper, such as Rees algebras and Rees algebras endowed with a suitable compatibility with differential operators. This will lead us to the definition of simple differential algebras, which will be essential for the definition of our invariants. 

The study of $n$-fold points of the hypersurface $X=V(f)$ is reformulated here in terms of the Rees algebra $\calo_{V^{(d)}}[fW^n]$. This is our first example of {\em simple} algebra. Attached to this Rees algebra is a well-defined \emph{differential algebra}.

Simple algebras which are differential will lead us naturally to the study of monic polynomials (\ref{eqintrof2}), where now $n=p^e$ is a power of the characteristic.

We also discuss here the notion of \emph{elimination algebras}. These are defined in terms of  differential algebras and transversal projections. Elimination algebras will play a central role in the definition of invariants. A first step in this direction will be given by our notion of \emph{$p$-presentation} in Definition \ref{pr_locp}.
 
We shall make use of a fundamental property of stability of transversality with monoidal transformations: To fix ideas set $X=\{f=0\}\subset V^{(d)}$ and a transversal projection $V^{(d)}\overset{\beta}{\longrightarrow}V^{(d-1)}$ as in (\ref{eqintrof2}). Consider now an arbitrary sequence of monoidal transformations
\begin{equation}\label{seqX}
\xymatrix@R=0pc@C=0pc{
X & & & & & X_1 &  & & & &  &  & & & &   X_r\\
V^{(d)}  &  & & & &   V_{1}^{(d)}\ar[lllll]_{\pi_{C_0}}   & & & & & \dots \ar[lllll]_{\pi_{C_1}} &  & & & &   V_{r}^{(d)}\ar[lllll]_{\pi_{C_{r-1}}}  \\
}
\end{equation}
where each  $X_{i+1}$ denotes the strict transform of $X_{i}$ and each $\pi_{C_{i}}$ is a monoidal transformation with center $C_{i-1}$ included in the $n$-fold points of $X_{i}$. The stability property of transversality is that (\ref{seqX}) induces a sequence 
\begin{equation}\label{seq2}
\xymatrix@R=0pc@C=0pc{
V^{(d-1)}  &  & & & &   V_{1}^{(d-1)}\ar[lllll]   & & & & & \dots \ar[lllll] &  & & & &   V_{r}^{(d-1)}\ar[lllll]}
\end{equation} 
together with projections $V_i^{(d)}\overset{\beta_i}{\longrightarrow}V_i^{(d-1)}$ which are transversal to $X_i$ along the $n$-fold points (the $\beta_i$ are defined in an open neighborhood of the $n$-fold points of $X_i$ in $V_i^{(d)}$). 

This will lead us to some form of \emph{transformations} of the monic polynomial in (\ref{eqintrof2}):
\begin{equation}\label{eqintrofi}
f^{(i)}(z_i)=z_i^n+a^{(i)}_1z_i^{n-1}+\dots+a^{(i)}_n\in\calo_{V_i^{(d-1)}}[z_i].
\end{equation}

The polynomials in (\ref{eqintrofi}) are not the strict transform of the first expression in (\ref{eqintrof2}). Changes of the transversal parameter $z_i$ would be required in the definition of each expression.

In Section \ref{sec222} sequences as (\ref{seqX})  are expressed as transformations of Rees algebras. In this context each transversal projection $\beta_i$ will define an elimination algebra on $V_i^{(d-1)}$. In this section, we also discuss a form of compatibility of elimination with monoidal transformations. This, in turn, will  lead to Theorem \ref{thm:BV} in which monomial algebras appear in a natural manner (\cite{BV3}). 

One of the objectives of this first Part is to assign a monomial algebra, say $\calo_{V_r^{(d)}}[I(H_1)^{h_1}\dots I(H_r)^{h_r}W^s]$  (see \ref{pi14} (1)),   to a sequence of transformations of $X$ as (\ref{seqX}). This sequence will be formulated here as a sequence of transformations of Rees algebras. This monomial algebra, assigned to (\ref{seqX}), will relate to the coefficients of $f^{(r)}(z_r)=z_r^n+a^{(r)}_1z_r^{n-1}+\dots+a^{(r)}_n\in\calo_{V_r^{(d-1)}}[z_r]$. In fact, we show that such expression can be chosen so each coefficient $a_i^{(r)}$ is divisible, in some weighted manner, by this monomial algebra (see Definition \ref{defmon}).

A first step in the definition of our inductive function $v-ord^{(d-1)}$ is addressed in Section \ref{sec333}, where a rational number is assigned to a $p$-presentation (slope at a point). A  notion of \emph{well-adapted} $p$-presentation at a point is introduced in Section \ref{sec555}.
It will be ultimately shown, in a further section, that the slope of $p$-presentations which are well-adapted at a point $x$, is $v-ord^{(d-1)}(x)$ (the value of the inductive function at $x$). This highlights the importance of this notion in what follows. 

Both Sections \ref{sec333} and  \ref{sec555} are focused in giving, in an explicit manner, the value of the Main Inductive function at a singular point. 

In Section \ref{sec666}, monoidal transforms of $p$-presentations are defined. This leads to the statement of the  two main results of this first Part: Main Theorems 1 and 2,  stated in Section \ref{secMT}.
Main Theorem 1 (Theorem \ref{result43}) asserts that the previously defined inductive function is independent of the chosen smooth projection $\beta$. Main Theorem 2 characterizes the monomial algebra, called here $\m_rW^s$, defined by the inductive functions. Proofs will be address in  Part \ref{partf}.
%


\vspace{0.25cm}

\noindent{\sf Part II: Strong monomial case}.

In Part I we have defined the inductive functions, $v-ord^{(d-1)}$, and a monomial algebra, say $\m_rW^s$, has been assigned to a sequence of transformations (\ref{seqX}). It can be shown that for any $n$-fold point
$$v-ord^{(d-1)}(x)\geq \ord(\m_rW^s)(x),$$
where the right hand side is a function defined for an arbitrary algebra (see (\ref{dford})). The function $\ord(\m_rW^s)$ is a nicely behaved upper semi-continuous function as opposed to the function in the left hand side.
The previous inequality between the previous functions will lead us to the numerical characterization of the strong monomial case,  expressed by the condition 
 $$v-ord^{(d-1)}(x)=\ord(\m_r W^s)(x),$$
for any $n$-fold point $x$.

It is proved in Theorem \ref{claim:gamma:res} that if such equality holds, then a combinatorial resolution of $\m_rW^s$ can be lifted to a simplification of the $n$-fold points. This settles \ref{pi14} (2).
%
%
%
\end{parrafo}

\begin{parrafo}
 
\vskip 3mm
\noindent{\bf Final comments.} The invariants studied in this paper make use of transversal projections $V^{(d)}\longrightarrow V^{(d-1)}$ and of elimination algebras defined in $V^{(d-1)}$. There are other approaches in the definition of invariants along $n$-fold points of a hypersurface.  The bibliography indicates some, but certainly not all the effort done in this way. An account on the problem, due to Hauser, appears in \cite{HauP}. There is an alternative approach of W\l odarczyk; his presentation in \cite{Wlo2} includes an important study of pathologies in positive characteristic. There are also recent contributions by Kawanoue-Matsuki (\cite{Kaw}, \cite{KM}), Hironaka (\cite{Hironaka06}), Cutkosky (\cite{Cut3}), and a fundamental contribution of Cossart-Jannsen-Saito in \cite{CJS} which proves embedded resolution for $2$-dimensional arithmetical schemes. We have profited from discussions with Encinas, Cossart, Hauser, Kawanoue, Lipman, Matsuki, Piltant  and from ideas of Ana Bravo which will be treated elsewhere.
\end{parrafo}

\begin{part}{Inductive functions and the tight monomial algebra.}

\section{Differential algebras, elimination and local presentations.}\label{sec111}

\begin{parrafo}
The initial motivation is the study of the highest multiplicity locus of an embedded hypersurface $X$. Here we begin in \ref{par1} by showing how to reformulate this study in terms of algebras. This reformulation will enable us to consider algebras with more structure. In fact, algebras with a form of compatibility with differential operators are studied in \ref{pq}, \ref{kr22} and \ref{rp23}, where the notion of \emph{absolute} and \emph{relative differential algebras} are discussed. 

It is in the context of differential algebras in which the fundamental notions of transversal projections and elimination algebras will be introduced (see \ref{nlsd} and \ref{csc}, respectively). 

The main objective of this section is to show that given a differential algebra, together with a transversal projection, the algebra can be entirely reconstructed in terms of two ingredients: 
\begin{enumerate}
\item the elimination algebra, and
\item a monic polynomial.
\end{enumerate}
This is the main result in this section, which is collected in Proposition \ref{pr_local}. This form of presentation of the algebra will be essential throughout this work. In the case of characteristic zero the monic polynomial can be chosen of degree one. In the case of positive characteristic one can choose the monic polynomial so as to have as degree a power of the characteristic. This will lead to the definition of $p$-presentations in Definition \ref{pr_locp}.

The particular feature of positive characteristic is played by the coefficients of this monic polynomial as will be shown in this development. The definition of the main invariant will rely entirely on these two ingredients.

\end{parrafo}

\begin{parrafo}\label{par1}{\bf Rees algebras and the resolution problem}. Here we introduce the notion of Rees algebras which will play a prominent role in our development.
Let $V^{(d)}$ be a smooth scheme over a perfect field $k$ of dimension $d$. The problem of resolution of singularities of a singular scheme embedded in $V^{(d)}$ can be stated in terms of Rees algebras over $V^{(d)}$. These are algebras of the form $\G=\bigoplus_{n\in\nat} I_nW^n$, where $I_0=\calo_{V^{(d)}}$ and each $I_n$ is a coherent sheaf of ideals. 
Here $W$ stands for a dummy variable introduced simply to keep track of the degree. It will be assumed that, locally at any point of $V$, $\G$ is a finitely generated $\calo_{V^{(d)}}$-algebra.

 A non-zero sheaf of ideals
 $J\subset \calo_{V^{(d)}}$  defines an upper-semi-continuous function $\nu(J): V^{(d)}\longrightarrow \mathbb Z,$ where 
 $\nu_x(J)$ denotes the order of the stalk $J_x$ at the local regular ring $(\calo_{V^{(d)},x}, m_x)$.
Recall that the order of $J$ in $\calo_{V^{(d)},x}$ is the highest integer $n$ so that $J_x\subset m_x^n$. The \emph{singular locus} of  $\G$ is the closed set
\begin{equation}\label{dfsing}
\Sing(\G)=\{x\in V^{(d)}\ |\ \nu_x(I_n)\geq n\hbox{ for each }n\in\nat\}.
\end{equation}
In the setting of \ref{pi12} in which $X=V(f)$, we will first attach to $X$ the algebra $\G=\calo_{V^{(d)}}[fW^n]$. The set $\Sing(\G)$ consists of the points of multiplicity $n$ of the hypersurface $X=V(f)$.

A monoidal transformation $V^{(d)}\overset{\pi_C}{\longleftarrow}V_1^{(d)}$ along the closed smooth center $C\subset\Sing(\G)$, 
defines a new Rees algebra, $\G_1=\bigoplus_{n\in\nat}I_n^{(1)}W^n$, called the \emph{transform of $\G$}. The transformation is denoted by
\vspace{-.1cm}
$$
\xymatrix@R=0pc@C=0pc{
\G & & & & & \G_1\\
V^{(d)}  &  & & & &   V_{1}^{(d)}\ar[lllll]_{\pi_C}
}$$

A sequence of transformations will be denoted by:
\vspace{-.1cm}
\begin{equation}\label{seqintro1}
\xymatrix@R=0pc@C=0pc{
\G & & & & & \G_1 &  & & & &  &  & & & &   \G_r\\
V^{(d)}  &  & & & &   V_{1}^{(d)}\ar[lllll]_{\pi_{C_1}}   & & & & & \dots \ar[lllll]_{\pi_{C_2}} &  & & & &   V_{r}^{(d)}\ar[lllll]_{\pi_{C_r}}  \\
}
\end{equation}
and herein we always assume that the exceptional locus of the composite morphism $V^{(d)}  \longleftarrow V^{(d)}_r$ is a union of hypersurfaces with only normal crossings.

The sequence (\ref{seqintro1}) is said to be a \emph{resolution of} $\G$ if $\Sing(\G_r)=\emptyset$. For $\G=\calo_{V^{(d)}}[fW^n]$, a resolution (\ref{seqintro1}) defines a simplification of $n$-fold points as in (\ref{opq})

A Rees algebra $\G$ is said to be \emph{simple at $x\in\Sing(\G)$} if there is an index $n\in\nat$ so that $\nu_x(I_n)=n$.  It is said to be \emph{simple} if this condition holds for any $x\in\Sing(\G)$. Such is the case for $\G=\calo_{V^{(d)}}[f_nW^n]$, when $f_n$ defines a hypersurface, say $X$, of maximum multiplicity $n$.
\end{parrafo}

\begin{parrafo}\label{pq}
Here $\beta:V^{(d)}\longrightarrow V^{(d-1)}$  will denote a smooth morphism of relative dimension one, from smooth schemes of dimensions $d$ and $d-1$, respectively. Throughout this paper these morphisms will be called \emph{projections}. Locally at a point $x\in V^{(d)}$, $V^{(d)}$
is \'etale over  $V^{(d-1)}\times \mathbb{A}^1$ (where $\mathbb{A}^1$ denotes the affine line), and such map is compatible with the projection on $V^{(d-1)}$ (\cite{AlKe}, p. 128). Consequently, the local ring $\calo_{V^{(d)},x}$ is \'etale over a localization of a polynomial ring in one variable, say $\calo_{V^{(d-1)},\beta(x)}[Z]$. After  restriction to a neighborhood of $x$, $Z$ gives rise to a global function at $V^{(d)}$, say $z$. So there is an inclusion $\calo_{V^{(d-1)}} [z]\subset \calo_{V^{(d)}} $, where $z \in \Gamma(\calo_{V^{(d)}}, V^{(d)}) $, and  the closed set $\{z=0\}$ is a section of $\beta:V^{(d)}\longrightarrow V^{(d-1)}$.

Given a ring $S[Z]$, a morphism of $S$-algebras, say $Tay: S[Z]\longrightarrow S[Z, T]$,
is defined by setting $Tay(Z)=Z+T$ (Taylor expansion). Here
$$Tay(f(Z))= f(Z+T)= \sum \Delta^{(r)}(f(Z))T^r,$$
for some operator $ \Delta^{(r)}\!: S[Z] \longrightarrow S[Z]$ defined from this morphism. It is well known that $\{\Delta^{(0)}, \Delta^{(1)}, \dots,\Delta^{(r)}\}$ is a basis of the free module of $S$-differential operators of order $r$. The same applies here for $\calo_{V^{(d-1)}} [z]$ if we assume that $\{dz\}$ is a basis of $\Omega_\beta^1$($=\Omega^1(\calo_{V^{(d)}}\ | \ \calo_{V^{(d-1)}})$). Namely,  $\{\Delta^{(0)}, \Delta^{(1)}, \dots,\Delta^{(r)}\}$ spans the sheaf of differential operators of order $r$ relative to the smooth morphism $\beta:V^{(d)}\longrightarrow V^{(d-1)}$.

Throughout this paper, we will slightly abuse the notation, here $\beta:V^{(d)}\longrightarrow V^{(d-1)}$ is called a {\em local projection} , and the function $z$ is said to be a 
{\em section} of $\beta$.

Let ${\mathcal G}=\bigoplus_{n\geq 0}I_nW^n$ be a Rees algebra on
a $d$-dimensional smooth scheme $V^{(d)}$. 
We always assume that $I_0=\calo_{V^{(d)}} $ and that ${\mathcal G}$ is locally a finite generated $\calo_{V^{(d)}} $-algebra. Namely that
$${\mathcal G}=\calo_{V^{(d)}}[f_{n_1}W^{n_1}, \dots ,f_{n_s}W^{n_s} ](\subset  \calo_{V^{(d)}}[W])$$ locally at any point of $V^{(d)}$.

Given two such algebras $\G_1$ and $\G_2$,  $\G_1\odot \G_2$ will denote the smallest algebra containing $\G_1$ and $\G_2$. In terms of local generators, if $\{f_1W^{n_1},\dots,f_rW^{n_r}\}$ generates $\G_1$ and $\{g_1W^{m_1},\dots,g_sW^{m_s}\}$ generates $\G_2$, then $\G_1\odot \G_2$ is generated by $\{f_1W^{n_1},\dots,f_rW^{n_r},g_1W^{m_1},\dots,g_sW^{m_s}\}$.

A function $\xymatrix@R=0pc @C=0.4pc{
\ord(\G)(-): V^{(d)}\ar[rr] & & \mathbb Q}$ is defined
\begin{equation}\label{dford} 
\ord(\G)(x)= \min_{n\geq0} \Big\{ \frac{\nu_x(I_n)}{n}  \Big\}
\end{equation} where $\nu_x$ denotes the order at the local regular ring 
$\calo_{V^{(d)},x}$. 
It takes only finitely many values. Note that the singular locus is
$\Sing(\G)=\{ x\in V^{(d)}\ |\ \ord(\G)(x)\geq 1  \}.$
\begin{remark}\label{kr22} It is a general fact that objects treated by resolution techniques are gathered in equivalence classes. Such is the case, for instance, with Log-resolutions of ideals on smooth schemes. If two ideals have the same integral closure, they undergo the same Log-resolution; so ideals are considered up to integral closure.
A similar situation applies here, where the objects are algebras: two algebras with the same integral closure will not be distinguishable. For instance, if $\G$ and $\G'$ are two algebras on  $V^{(d)}$ with the same integral closure, then they define the same functions $\ord(\G)=\ord(\G')$ (in particular, $\Sing(\G)=\Sing(\G')$, Proposition 4.4, \cite{VV3}). The reader should keep aware of this fact, as it also affects the notation. The expression $fW^t \in {\mathcal G}=\bigoplus_{n\geq 0}I_nW^n$ 
means that $f^r\in I_{t\cdot r}$ for some positive integer $r$.

A Rees algebra can be defined by fixing an ideal $I$ and a positive integer $s$, say $\calo_V[IW^s](\subset \calo_V[W])$, which  we denote simply as $IW^s$. Moreover, up to integral closure, any Rees algebra is of this kind (Remark 1.3 \cite{EncVil06}). In this case,  $fW^t \in \calo_V[IW^s]$ means that $f^s\in I^{t}$.
\end{remark}
\begin{parrafo}\label{rp23}
An algebra  ${\mathcal G}=\bigoplus_{n\geq 0}I_nW^n$ over  $V^{(d)}$ is said to be a {\em differential algebra} if $D_{r}(I_n)\subset I_{n-r}$ for any $r< n$ and for any differential operator $D_r$  of order $r$, whenever we restrict to an affine open subset of $V^{(d)}$.

$\G$ is said to be an \emph{absolute differential algebra}, if this property holds for all $k$-linear differential operators. When a smooth projection 
 $V^{(d)}\overset{\beta}{\longrightarrow}  V^{(d-1)}$ is fixed and the previous property holds 
for differential operators which are $\calo_{ V^{(d-1)}}$-linear, or say, 
$\beta$-relative operators, then $\G$ is said to be a \emph{$\beta$-relative differential algebra}, or simply \emph{$\beta$-differential}.

If $\G$ is an absolute differential algebra, then it is also a $\beta$-relative differential algebra {\em for any} smooth morphism $V^{(d)}\overset{\beta}{\longrightarrow} V^{(d-1)}$ defined over $k$. The $\beta$-relative structure has an advantage:
The transform of an absolute differential algebra is not absolute differential, but the notion of $\beta$-differential algebra will turn out to be well suited with transformations.

If $\G$ is not a differential algebra, then it has a natural extension to 
a differential algebra (Theorem 3.4, \cite{VV3}). The same holds if $\G$ is not a 
$\beta$-differential algebra. These natural extensions are compatible with integral closure: if $\G_1$ and $\G_2$ have the same integral closure, the same holds for their extensions to differential algebras or to $\beta$-differential algebras (Theorem 6.14, \cite{VV3}).
\end{parrafo}
\begin{remark}
When $\G$ is a $\beta$-differential Rees algebra, then, locally, there is a finite set of elements of $\G$, say $\{f_{n_1}W^{n_1},\dots,f_{n_s}W^{n_s}\}$, so that 
$$\G=\calo_{V^{(d)}}[f_{n_i}W^{n_i},\Delta^{(\alpha_i)}(f_i)W^{n_i-\alpha_i}]_{1\leq\alpha_i\leq n_i-1,\ 1\leq i\leq s},$$
with $\Delta^{(\alpha_i)}$ as in \ref{pq}. Conversely, these local presentations characterize $\beta$-differential algebras (Theorem 2.9, \cite{VV3}).
\end{remark}

\end{parrafo}
%
%

\begin{parrafo}\label{nlsd}{\bf Transversal projections}. 
 The graded algebra of the  maximal ideal $m_x$ of a point $x\in V^{(d)}$,
say $\mbox{Gr}_{x}({\mathcal O}_{V^{(d)},x})$,  is isomorphic to a
polynomial ring. When $x$ is a closed point, it is a polynomial ring in $d$-variables, which is the coordinate ring
associated to the tangent space of $V^{(d)}$ at $x$, namely
$\mbox{Spec}(\mbox{Gr}_{x}({\mathcal O}_{V^{(d)},x}))={\mathbb
T}_{V^{(d)},x}$. 
The {\em initial ideal} or {\em tangent ideal} of
${\mathcal G}$ at $x\in \mbox{Sing }{\mathcal G}$, say ${\mbox{In}}_{x}({\mathcal G})$, is the
ideal of $\mbox{Gr}_{x}({\mathcal O}_{V^{(d)},x})$ generated by the
elements $\mbox{In}_x(I_n)$ for all $n\geq 1$, where $\In_x(I_n)$ is the class of $I_n$ at $m_x^n/m_x^{n+1}$. Observe that
${\mbox{In}}_{x}({\mathcal G})$ is zero unless $\mbox{ord}({\mathcal G})(x)=1$.  The zero set of the tangent ideal $\In_x(\G)$ in $\mbox{Spec
}(\mbox{Gr}_{x}({\mathcal O}_{V^{(d)},x}))$ is the {\em tangent cone} of
${\mathcal G}$ at $x$, denoted by ${\mathcal C}_{{\mathcal G},x}$.

Given a vector space $\mathbb{V}$, a vector $v\in \mathbb{V}$ defines a translation, say $tr_v(w)=w+v$ for $w\in\mathbb{V}$. There is a largest linear subspace, denoted by $\L_{\G,x}$, so that $\C_{\G,x}$ is invariant under translations of $\L_{\G,x}$, that is, $tr_v(\C_f)=\C_f$ for any $v\in\L_f$. This subspace $\L_{\G,x}$ is called the \emph{ linear space of vertices}.
\end{parrafo}

\begin{definition}(Hironaka's $\tau$-invariant). 
$\tau_{{\mathcal G},x}$
will denote the minimum number of variables required to express generators of the tangent ideal $\In_x(\G)$. This algebraic definition can be reformulated geometrically: $\tau_{\G,x}$ is the codimension of the linear subspace $\L_{\G,x}$ in $\mathbb{T}_{V^{(d)},x}$.
\end{definition}

\begin{parrafo}

Fix now a closed point $x\in V^{(d)}$. Let $V^{(d)}\overset{\beta}{\longrightarrow} V^{(d-1)}$ be smooth and set 
$\beta(x)=y\in V^{(d-1)}$. A regular system of parameters $\{ y_1, \dots , y_{s}\}$ in $\calo_{V^{(d-1)},y}$,   extends to $\{ y_1, \dots , y_{s},z\}$, a regular system of parameters in $\calo_{V^{(d)},x}$. Here $x$ is a point of
$\beta^{-1}(y)$, and the tangent space of this subscheme at $x$, say 
${\mathbb
T}_{\beta^{-1}(y),x}$, is identified with the subscheme in ${\mathbb
T}_{V^{(d)},x}$ defined by the linear forms 
$\langle {\mbox{In}}_{x}(y_1), \dots, 
{\mbox{In}}_{x}(y_{s}) \rangle \subset \mbox{Gr}_{x}({\mathcal O}_{V^{(d)},x})$ 
(i.e., a one dimensional subspace in   $\mathbb{T}_{V^{(d)},x}$).

\begin{definition} \label{trpt} 
A local projection $\beta:V^{(d)}\longrightarrow V^{(d-1)}$ is said to be {\em transversal} to 
$\G$ at  $x\in
\mbox{Sing }({\mathcal G})$  if ${\mathcal C}_{{\mathcal G},x}\cap {\mathbb
T}_{\beta^{-1}(y),x}= \mathds{O}$, the origin of ${\mathbb
T}_{V^{(d)},x}$. 
The local projection is said to be 
{\em transversal} to 
$\G$ if it is so at any point of $
\mbox{Sing }{\mathcal G}$. 
Transversality is an open condition so we are led to consider this condition only at closed points (see Remark 8.5 in \cite{BV3}).
\end{definition}

\end{parrafo}

\begin{parrafo}\label{csc}{\bf Elimination algebras}.
Set a local projection $\beta:V^{(d)}\longrightarrow V^{(d-1)}$. Let $x\in \Sing(\G)$ be a closed point in $V^{(d)}$,  so $y=\beta(x)$ is closed in $ V^{(d-1)} $. A regular system of parameters $\{ y_1, \dots , y_{d-1}\}\subset\calo_{V^{(d-1)},y}$ extends to  a regular system of parameters $\{ y_1, \dots , y_{d-1},z\}$ in $\calo_{V^{(d)},x}$. In this case, $z$ defines a section of $\beta:V^{(d)}\longrightarrow V^{(d-1)}$ after suitable restrictions.

Take $\G$ to be a simple algebra, and let 
$\beta:  V^{(d)}\longrightarrow  V^{(d-1)}$ be transversal to $\G$.
 Fix a closed point $x \in \Sing(\G)$. The Weierstrass Preparation Theorem  ensures that, taking restrictions in \'etale topology, $\G$ has the same integral closure as an algebra $\calo_{V^{(d)}}[f_1(z)W^{n_1}, \dots ,f_s(z)W^{n_s} ]$, where each
 \begin{equation}\label{qeq22}
f_i(z)=z^{n_i}+a^{(i)}_{n_i-1}z^{n_i-1}+\cdots +a^{(i)}_0\in \calo_{V^{(d-1)}}[z]
\end{equation}
is a monic polynomial of degree $n_i\in\mathbb{Z}_{\geq0}$ (see 4.7 in \cite{VV4}). 

The following properties are known to hold within this setting:

{\bf P0)} the restriction of $\beta$ to $\Sing(\G)$, say $\beta:\Sing(\G)\longrightarrow \beta(\Sing(\G))$, is a set theoretical bijection and two corresponding points have the same residue fields. Namely, $k(x)\cong k(\beta(x))$ (1.15 and Theorem 4.11  \cite{VV4}, or 7.1 \cite{BV3}).

If $\G$ is a $\beta$-relative differential algebra, then a Rees algebra on $V^{(d-1)}$, say
$\R_{\G,\beta}\subset \calo_{V^{(d-1)}}[W]$, is defined. It is called the {\em  elimination algebra of $\G$}, and has the following properties:

{\bf P1)}  $\beta(\Sing(\G)) \subset \Sing(\R_{\G,\beta})$; moreover if $C$ is a closed and smooth scheme included in $\Sing(\G)$, then 
$\beta(C) (\subset V^{(d-1)})$ is smooth, isomorphic to $C$, and 
$\beta(C) \subset \Sing(\R_{\G,\beta})$  (Theorem 9.1 \cite{BV3}).

{\bf P2)} (Theorem 5.5, \cite{VV4})
Fix two projections: 
$$\xymatrix@R=0pc@C=0pc{
& & \G\\
& & V^{(d)}\ar[llddddd]_{\beta\!\!}\ar[rrddddd]^{\!\!\beta'}\\
\\
\\
\\
\\
 V^{(d-1)} & & & & V'^{(d-1)}\\
\R_{\G,\beta}  & & & & \R_{\G,\beta'}
}$$
where both $\beta$ and $\beta'$ are transversal to $\G$. This defines an algebra $\R_{\G,\beta}$ over $V^{(d-1)}$ and an algebra $\R_{\G,\beta'}$ over $V'^{(d-1)}$. At any point $x\in\Sing(\G)$,
$$\ord(\R_{\G,\beta})(\beta(x))=\ord(\R_{\G,\beta'})(\beta'(x)). $$

{\bf P3)} (1.15 \cite{VV4}) If $\ord(\R_{\G,\beta})(y)>0$ at a point $y\in V^{(d-1)}$, the restriction of (\ref{qeq22}) 
to the fiber over $y$ (to $\beta^{-1}(y)$), say 
$$\overline{f_i}(Z)=Z^{n_i}+\overline{a}^{(i)}_{n_i-1}Z^{p^e-1}+\dots+\overline{a}^{(i)}_{0}\in k(y)[Z];$$ 
 is a power of a purely inseparable polynomial. Namely, 
 $\overline{f_i}(Z)=(Z^{p^{r_i}}+ b_i)^{m_i}$ at $k(y)[Z]$.
 Moreover, there is at most one point $ x\in V^{(d)}$ so that 
 $\beta(x)=y$ and $\ord(\G)(x)>0$.

A particular feature of characteristic zero is that $z$ can be chosen to be of maximal contact. This is not always the case in positive characteristic, and the relative differential structure will partially fill in this gap. The previous formulation, in which the algebra is generated by monic polynomials, holds locally. In this local description,  $z$ is a section of $\beta:V^{(d)}\longrightarrow V^{(d-1)}$ as defined in \ref{pq}.
\end{parrafo}

\begin{proposition}\label{pr_local}{\bf (Local presentation)} Set  $x\in\Sing(\G)$ a closed point and  $V^{(d)}\overset{\beta}{\longrightarrow} V^{(d-1)}$ transversal to $\G$ at $x$. Assume that $\G$ is a $\beta$-relative differential algebra, that there is an element $f_nW^n\in\G$, $f_n$ of order $n$ at  $\calo_{V^{(d)},x}$, and that $f_n=f_n(z)$ is a monic polynomial of degree $n$ in $\calo_{V^{(d-1)},\beta(x)}[z]$, where $z$ is a $\beta$-section and an element at  $\calo_{V^{(d)},x}$.  Then, at a neighborhood of $x$, $\G$ has the same integral closure as
\begin{equation}\label{eqlocalpresen}
\calo_{V^{(d)}}[f_n(z)W^n,\Delta^{(\alpha)}(f_n(z))W^{n-\alpha}]_{1\leq \alpha\leq n-1}\odot\R_{\G,\beta},
\end{equation}
where $\R_{\G,\beta}$ is identified with $\beta^*(\R_{\G,\beta})$, and $\Delta^{(\alpha)}$ are as in {\rm \ref{pq}}.
Moreover, $\R_{\G,\beta}$ is non-zero whenever $\Sing(\G)$ is not  of co-dimension one locally at $x$.
\end{proposition}

\begin{proof}
The last assertion follows from Theorem 4.11 i) \cite{VV4}.
Take $f_n(z)W^n\in \{f_{1}W^{n_1},\dots,f_{s}W^{n_s}\}$ as in (\ref{qeq22}). For ease of notation we consider the case $s=2$, i.e., $\G=\calo_{V^{(d)}}[f_n(z)W^n,g_m(z)W^m]$.

We follow here the arguments and notation as in Chapter 1 in 
\cite{VV4}, particularly Prop.1.29.
Rees algebras are endowed with a natural graded structure. Elimination algebras are also Rees algebras. They are defined as a specialization of the so called {\em universal elimination algebras}, which are graded subalgebras in a polynomial ring.

Take variables $Z$, $Y_1,\dots,Y_n$ and $V_1,\dots,V_m$ over a field $k$, and set
$$F_n(Z)=(Z-Y_1)\cdot(Z-Y_2)\dots(Z-Y_n).$$
This is the so called \emph{universal  polynomial of degree $n$}, and  $f_n=f_n(z)$ can be obtain as a specialization of  $F_n(Z)$. Similarly, let
$$G_m(Z)=(Z-V_1)\cdot(Z-V_2)\dots(Z-V_m)$$
be the universal polynomial of degree $m$ which will specialize to $g_m(z)$.

The natural action of the permutation groups $\S_n$ on $k[Y_1,\dots,Y_n]$, and of $\S_m$ on $k[V_1,\dots,V_m]$, induces an action of the product $\S_n\times \S_m$  on $k[Z,Y_1,\dots,Y_n,V_1,\dots,V_m]$ by fixing $Z$. This group also acts on the subring
$$S=k[Z-Y_1,Z-Y_2,\dots,Z-Y_n,Z-V_1,Z-V_2,\dots,Z-V_m].$$
The subring of invariants of $S$, say $S^{\S_n\times\S_m}$, is
$$k[\Delta^{(\alpha)}(F_n(Z)),\Delta^{(\alpha')}(G_m(Z))]_{0\leq \alpha\leq n-1,\ 0\leq \alpha'\leq m-1},$$
where $\Delta^{(\alpha)}(F_n(Z))$ is an homogeneous polynomial of degree $n-\alpha$, obtained as in \ref{pq}. Similarly $\Delta^{(\alpha')}(G_m(Z))$ is homogeneous of degree $m-\alpha'$. 

As this actions are linear, $S^{\S_n\times\S_m}$ inherits the grading of the polynomial ring $k[Z,Y_i,V_j]$. We add a dummy variable $W$ that will simply express the degree of each homogeneous element. Hence, the subring of invariants $S^{\S_n\times\S_m}$ is now
$$k[\Delta^{(\alpha)}(F_n(Z))W^{n-\alpha},\Delta^{(\alpha')}(G_m(Z))W^{m-\alpha'}]_{0\leq \alpha\leq n-1,\ 0\leq \alpha'\leq m-1}.$$

Consider the subring
$$S'=k[(Z-Y_2)-(Z-Y_1),\dots,(Z-Y_n)-(Z-Y_1),(Z-V_1)-(Z-Y_1),\dots,(Z-V_m)-(Z-Y_1)],$$ 
of $S$. Note that $\S_n\times\S_m$ acts on $S'$. The universal elimination algebra is, in this case of $s=2$, defined as the invariant ring $S'\,^{\S_n\times\S_m}$.

The key observation to prove the assertion is that  $S$ is spanned by two subrings: $k[Z-Y_1,\dots,Z-Y_n]$ and $S'$, and $\S_n\times\S_m$ acts on both. 

Recall that the subring of invariants in the first is 
$T=k[F_n(Z))W^n, \Delta^{(\alpha)}(F_n(Z))W^{n-\alpha}]_{1\leq \alpha\leq n-1}$,
and the one of the second is the universal elimination algebra, say $R(\subset S')$.  

Thus both invariant algebras, $T$ and $R$, are included in $S^{\S_n\times\S_m}$. Let $T\odot R$ denote the smallest algebra containing both rings.  We now claim that $T\odot R\subset S^{\S_n\times S_m}$ is a finite extension of graded subalgebras of $S$.
In order to prove this last assertion note that $S$ is a finite extension of both subalgebras. 

The statement follows now from the previous observation. In fact, $\G$ and (\ref{eqlocalpresen}) are obtained by specialization of the previous subrings. This specialization preserves the grading. On the other hand, integral extension of rings are preserved by specialization (change of base rings).
\end{proof}

\begin{remark}
Fix a Rees algebra ${\mathcal G}=\bigoplus_{n\geq 0}I_nW^n$. If the setting of Proposition \ref{pr_local} holds at a closed point $x\in \Sing(\G)$, then it holds globally after taking suitable restrictions of  $V^{(d-1)}$ to a neighborhood of $\beta(x)$, and of $V^{(d)}$ to a neighborhood of $x$. Moreover, $z$ defines a $\beta$-section.

If the characteristic is zero $I_1$ has order one at $\calo_{V^{(d)},x}$, and $z\in I_1$ can be chosen as an element of order one at this local ring. This is not always the case  in positive characteristic. However, as $\G$ is a simple $\beta$-relative differential algebra, one can check that there is a power of the characteristic, say $p^e$,  so that $I_{p^e}$ has order $p^e$ at $\calo_{V^{(d)},x}$. Therefore the integer $n$ in the last Proposition can be chosen as a power of the characteristic. This integer $p^e$ is defined in terms of $\G$ and the closed point $x\in \Sing(\G)$. This leads to: 
\end{remark}
\begin{definition}\label{pr_locp}({\bf $p$-Presentations}).
Fix, after suitable restriction in \'etale topology, a projection $V^{(d)}\overset{\beta}{\longrightarrow} V^{(d-1)}$ transversal to a simple $\beta$-relative differential Rees algebra $\G$. Assume that $\Sing(\G)$ has no components of co-dimension one.

Assume also that:

\begin{enumerate}
\item[i)] There is a $\beta$-section $z$, a global function on $V^{(d)}$, and $\{dz\}$ is a basis of $\Omega_\beta^1$

\item[ii)] There is an element $f_{p^e}(z)W^{p^e}\in \G$, where $f_{p^e}(z)$ is a monic polynomial of order $p^e$, say
$$f_{p^e}(z)=z^{p^{e}}+a_1z^{p^{e}-1}+\dots+a_{p^{e}}\in \calo_{V^{(d-1)}}[z],$$
where each $a_i$ is a global function on $V^{(d-1)}$.
\item[iii)] The conditions in (\ref{eqlocalpresen}) holds for $\G$  and 
\begin{equation} \label{eqdpl} 
\calo_{V^{(d)}}[f_{p^e}(z)W^{p^{e}},\Delta^{(\alpha)}(f_{p^e}(z))W^{p^{e}-\alpha}]_{1\leq \alpha \leq p^{e}-1}\odot\beta^*(\R_{\G,\beta}).
\end{equation}
That is, $\G$ and (\ref{eqdpl}) have the same integral closure.
\end{enumerate}

In this case, we say that $\beta:V^{(d)}\longrightarrow V^{(d-1)}$, the $\beta$-section $z$,  and $f_{p^e}(z)=z^{p^{e}}+a_1^{}z^{p^{e}-1}+\dots+a_{p^{e}}^{}$ define a \emph{$p$-presentation} of $\G$. These data will be denoted by:
\begin{equation}\label{eqld}
p\P(\beta:V^{(d)}\longrightarrow V^{(d-1)},  z,  f_{p^e}(z)=z^{p^{e}}+a_1^{}z^{p^{e}-1}+\dots+a_{p^{e}}^{}),
\end{equation}
or simply $p\P(\beta,  z,  f_{p^e}(z))$. Clearly (\ref{eqdpl}) is expressed only in terms of $\R_{\G,\beta}$ and $p\P(\beta,  z,  f_{p^e}(z))$.

\end{definition}
\vspace{0.15cm}

 \section{{Monomial algebras and the behavior of elimination under monoidal transformations.}}\label{sec222}

\vspace{0.15cm}
 
\begin{parrafo} The definition of elimination algebras makes use of the notion of the relative differential structure. We now discuss some results that grow from a form of  compatibility of the relative differential structure with monoidal transformations. 

Recall that a {\em sequence of transformations} of  $\G$ is a concatenation of transformations 
\begin{equation}\label{unaseq}
\xymatrix@R=0pc@C=0pc{
\G & & & & & \G_1 &  & & & &  &  & & & &   \G_r\\
V^{(d)}  &  & & & &   V^{(d)}_{1}\ar[lllll]_{\pi_0}   & & & & & \dots \ar[lllll]_{\pi_1} &  & & & &   V^{(d)}_{r}\ar[lllll]_{\pi_{r-1}}  
}
\end{equation}
where we always assume that the exceptional locus of  $V^{(d)}\overset{\pi}{\longleftarrow}V_r^{(d)}$ is a union of hypersurfaces with normal crossings. In the first part of this section we study the compatibility of transversality and elimination algebras with monoidal transformations. Sequences as (\ref{unaseq}) will also give rise to the definition of the so called monomial algebras (Definition \ref{def123}), and to a notion of monomial contact introduced in Definition \ref{defmon}. This notion appears in the formulation of Main Theorem 2.
\end{parrafo}
\begin{parrafo}
Transversal projections are defined for simple algebras. When  $\G$ is a simple algebra, we claim that all the $\G_i$ defined in (\ref{unaseq}) are also simple. It suffices to check this property locally. Fix a closed point $x\in C\subset\Sing(\G)$, where $C$ is a smooth center. There is an integer $n$ and an element $f_n\in I_n$ so that $\nu_x(f_n)=n$. Note that $\nu_C(f_n)=n$ and $f_n$ is equimultiple at $C$ locally at $x$, so the strict transform of $f_n$ has multiplicity at most $n$ on points on the exceptional locus, and hence $\G_1$ is simple.

Take $\G$ to be a simple algebra on $V^{(d)}$, together with a transversal projection $\beta:V^{(d)}\longrightarrow V^{(d-1)}$. Assume that $\G$ is a $\beta$-relative differential algebra. A notion of compatibility of this properties with monoidal transformations can be formulated as follows (\cite{BV3}):

After suitable restrictions to an \'etale cover of  $V^{(d)}$, the sequence (\ref{unaseq}) induces a diagram
\begin{equation}\label{cuadrado}
\xymatrix@R=0pc@C=0pc{
\G & & & & & \G_1 &  & & & &  &  & & & &   \G_r\\
V^{(d)}\ar[ddd]^\beta  &  & & & &   V^{(d)}_{1}\ar[lllll]_{\pi_0}\ar[ddd]^{\beta_1}    & & & & & \dots \ar[lllll]_{\pi_1} &  & & & &   V^{(d)}_{r}\ar[lllll]_{\pi_r}\ar[ddd]^{\beta_{r-1}} \\
\\
\\
V^{(d-1)}  &  & & & &   V^{(d-1)}_{1}\ar[lllll]_{\pi'_0}   & & & & & \dots \ar[lllll]_{\pi'_1} &  & & & &   V^{(d-1)}_{r}\ar[lllll]_{\pi'_{r-1}}\\
\R_{\G,\beta} & & & & & (\R_{\G,\beta})_1 &  & & & &  &  & & & &   (\R_{\G,\beta})_r
}\end{equation}
where:
\begin{enumerate}
\item[(i)] Each vertical morphism $\beta_i:V_i^{(d)}\longrightarrow V^{(d-1)}_i$ is transversal to $\G_i$,  and
each $\G_i$ is a $\beta_i$-differential algebra. 
These $\beta_i$ are defined only in an neighborhood of $\Sing(\G_i)$.
\item[(ii)] The lower sequence induces transformations of the elimination algebra $\R_{\G,\beta}$, and furthermore, each $(\R_{\G,\beta})_i$ is the elimination algebra of $\G_i$ relative to $\beta_i:V_i^{(d)}\longrightarrow V_i^{(d-1)}$, that is, $(\R_{\G,\beta})_i=\R_{\G_i,\beta_i}$ (Theorem 9.1 \cite{BV3}).
\end{enumerate}

\begin{definition}\label{def:rtrans}
A smooth morphism $V^{(d)}_r\overset{\beta_r}{\longrightarrow} V^{(d-1)}_r$ is said to be \emph{$r$-tranversal} to $\G_r$ if there is a transversal morphism $V^{(d)}\overset{\beta}{\longrightarrow}V^{(d-1)}$, as in Definition \ref{trpt}, and a simple $\beta$-differential algebra $\G$ over $\calo_{V^{(d)}}$, so that $\G_r$ and $\beta_r$ arise from a diagram as that in (\ref{cuadrado}).
\end{definition}

\begin{remark}
In characteristic zero, given a simple differential algebra $\G$, there are hypersurfaces of maximal contact at $V^{(d)}$. We fix one such hypersurface, and given a sequence of transformations of $\G$ (\ref{unaseq}), we consider the strict transforms of that fixed hypersurface. 
Here hypersurfaces of maximal contact will be replaced by transversal projections. We shall fix a transversal projection at $V^{(d)}$ and for any sequence (\ref{unaseq}) we will make use of the lifting of this fixed projection in (\ref{cuadrado}).

Local $p$-presentations of $\G_r$ will be defined in terms of $\beta_r$, where $\beta_r$  arises from the fixed smooth transversal morphism $\beta$.
In Section \ref{sec666},  a notion of transformation of $p$-presentations will be defined. This together with the theorems on Section \ref{sec555} will show that given a simple algebra $\G$, if $V^{(d)}$ can be covered by $p$-presentations of the form
$p\P(\beta,  z,  f_{p^e}(z)=z^{p^{e}}+a_1^{}z^{p^{e}-1}+\dots+a_{p^{e}}^{}),$
with the same exponent $p^e$, then the same holds for $\G_r$ at $V_r^{(d)}$. Namely,  that there is a covering of $V_r^{(d)}$ by presentations of the form $p\widetilde{\P}(\beta_r,  \widetilde{z},  \widetilde{f}_{p^e}(\widetilde{z})=\widetilde{z}^{p^{e}}+\widetilde{a}_1\widetilde{z}^{p^{e}-1}+\dots+\widetilde{a}_{p^{e}}),$
where $\beta_r$ is $r$-transversal, with the same exponent $p^e$ on any such $p$-presentation.
\end{remark}

\begin{definition}\label{def123} Let $E=\{H_1,\dots, H_r\}$  be a set of smooth hypersurfaces with normal crossings. A \emph{monomial ideal} supported on $E$ is an invertible sheaf of ideals of the form $\m=I(H_1)^{\alpha_1}\cdot I(H_2)^{\alpha_2}\cdots I(H_r)^{\alpha_r}$, for some integers $\alpha_i\geq0$.

 A \emph{monomial algebra} will be a Rees algebra of the form $\calo_{V}[\m W^s]$ for some positive integer $s$. This algebra will be denoted by $\m W^s$. 
\end{definition}

\begin{remark}\label{rmkintM}
\ \ (1) Fix a monomial algebra $\m W^s=\calo_{V}[\m W^s]$. Locally at a point $x\in V$, $\m_x$ is the ideal spanned by a monomial on a regular system of parameters of $\calo_{V^{(d)},x}$. Recall that Rees algebras are to be considered up to integral closures. Given $f_n\in 
\calo_{V^{(d)},x}$, $f_nW^n\in \m W^s$ if and only if $f_n^s$ is divisible by 
$\m_x^n$ at $\calo_{V^{(d)},x}$ for any $x$ at $V^{(d)}$.

\vspace{0.2cm}

\noindent (2) Assume that $\m W^s$ is the Rees algebra generated by the monomial ideal $\m=I(H_1)^{h_1}\dots I(H_r)^{h_r}$ at degree $s$. Rees algebras are considered up to integral closure, so we shall now describe the integral closure of $\m W^s$, say $\widetilde{\m W^s}=\bigoplus J_tW^t$. Given a positive integer $t$, the ideal corresponding to the degree $t$, say $J_t$, is generated by a monomial, say 
\begin{equation}\label{eqintM}
\m^{[t]}=J_t=I(H_1)^{\lfloor\frac{h_1t}{s}\rfloor}\dots I(H_r)^{\lfloor\frac{h_rt}{s}\rfloor}.
\end{equation}
Moreover,
$\ord(\m W^s)(x)\leq \ord(\m^{[t]} W^t)(x),$
and  equality holds if and only if $\m W^s$ and $\m^{[t]} W^t$ have the same integral closure.
\end{remark}

\begin{parrafo}
Let $\pi: V'\longrightarrow V$ be a smooth morphism, then the pull-backs of the hypersurfaces of $E$ have normal crossings at $V'$ and a monomial ideal supported on $E$ has a natural lifting to $V'$.

In our setting, we fix a transversal smooth morphism $\beta: V^{(d)} \longrightarrow V^{(d-1)}$ as in Definition \ref{trpt}, a sequence (\ref{unaseq}) induces a diagram (\ref{cuadrado})
with smooth morphisms $\beta_i$ defined in a neighborhood of $\Sing(\G_i)$. Note that at each such neighborhood, the exceptional hypersurfaces 
in $V^{(d)}_{i}$ are pull-backs of the exceptional hypersurfaces at 
$V^{(d-1)}_{i}$. In particular, a monomial algebra supported on the exceptional locus of the composite map $V^{(d-1)}\longleftarrow V^{(d-1)}_r$, say
\begin{equation}\label{ecdm}
\m_rW^s=I(H_1)^{h_1}\dots I(H_r)^{h_r}W^s,
\end{equation}
can be naturally lifted to a monomial algebra supported on the exceptional locus of $V^{(d)}\longleftarrow V_r^{(d)}$.

\begin{theorem}[Bravo-Villamayor \cite{BV3}]\label{thm:BV}

Let $\G$ be a simple differential algebra and assume that $\Sing(\G)$ has no component of co-dimension one. There is a sequence of  transformations (\ref{unaseq}),
so that for any local transversal projection 
 $\beta:V^{(d)}\longrightarrow V^{(d-1)}$ (defined by restriction to an \'etale coverring of $V^{(d)}$), the induced sequence $(\ref{cuadrado})$ is such  that $(\R_{\G,\beta})_r$ is a monomial algebra supported on the exceptional locus. Furthermore, the monomial algebra $\beta^{*}_r((\R_{\G,\beta})_r)$ is independent of $\beta$.
\end{theorem}

In what follows, we can take, \'etale locally, a sequence (\ref{cuadrado})
as in the formulation of the Theorem \ref{thm:BV} (Main Theorem in \cite{BV3}). So here, $(\R_{\G,\beta})_r\subset \calo_{V^{(d-1)}}[W]$ is monomial and supported on the exceptional locus, and so is its pull-back to $V_r^{(d)}$. The same holds if we enlarge the sequence of transformations as this condition is stable.

We identify $(\R_{\G,\beta})_r$ with its pull-back, say 
$$(\R_{\G,\beta})_r=I(H_1)^{\alpha_1}\dots I(H_r)^{\alpha_r}W^s=\mathcal{N}W^s.$$
 It will be shown that locally at any closed point of $\Sing(\G_r)$, there is a $\beta_r$-section $z'$, a monic polynomial, say $ f_{p^e}^{(r)}$, so that $\G_r$ has the same integral closure as:
$$\calo_{V_r^{(d)}}[f_{p^e}^{(r)}(z)W^{p^e},\Delta^\alpha(f_{p^e}^{(r)}(z))W^{p^e-\alpha}]_{1\leq \alpha \leq p^e-1}\odot \mathcal{N}_rW^s$$
\end{parrafo}
\begin{parrafo}\label{trdpp}
The outcome of Theorem \ref{thm:BV}, in the case of fields of characteristic zero, is known as the reduction to the monomial case. In that context it is simple to extend (\ref{unaseq}) to a resolution. This is not the case in positive characteristic, however the following definition will lead us to the study of the role of the exceptional divisors.
\end{parrafo}
\begin{definition}\label{defmon}

\begin{enumerate}
\item[1)] We say that a monomial algebra $\m_r W^s$  (\ref{ecdm}) has \emph{monomial contact} with $\G_r$ if locally at any closed point $x\in\Sing(\G_r)$ there is a $\beta_r$-section  $z$ of order one at $\calo_{V_r^{(d)},x}$, so that 
$$\G_r\subset\langle z \rangle W \odot \m_r W^s.$$

\item[2)] A local $p$-presentation  of $\G_r$, say  $p\P(\beta_r, z,  f^{(r)}_{p^e}(z))$ (with$ f_{p^e}^{(r)}=z^{p^{e}}+a_1z^{p^{e}-1}+\dots+a_{p^{e}})$),
 is said to be \emph{compatible with the monomial algebra} $\calo_{V^{(d-1)}}[\m_r W^s]$ locally at $x\in\Sing(\G_r)$
 if the previous condition holds for the $\beta_r$-section $z$. Proposition \ref{pr_locp} ensures that this is equivalent to two conditions:
\begin{enumerate}
\item[i)] $(\R_{\G,\beta})_r\subset \calo_{V_r^{(d-1)}}[\m_r W^s]$,

\item[ii)] $a_i W^i \in \calo_{V_r^{(d-1)}}[\m_r W^s]$, for $1 \leq i \leq p^e$ (Remark \ref{kr22}).
\end{enumerate}
\end{enumerate}

\end{definition}
\end{parrafo}
 
\begin{parrafo}\label{p26}
We will show that given a simple algebra $\G$ and a sequence of transformations as in 
(\ref{unaseq}),  there is a monomial algebra $\m_r W^s$ supported on the exceptional locus which has monomial contact with $\G_r$. That is, locally at any point $x\in\Sing(\G_r)$ there is a $\beta_r$-section $z$ of order one at $x$, so that $\G_r\subset\id{z}W\odot \m_r W^s.$ Main Theorem 2 will show that this monomial algebra will be defined in terms of the sequence (\ref{unaseq}), with independence of the choice of $\beta$ (of (\ref{cuadrado})).
\end{parrafo}

 \section{{Invariants defined in terms of $p$-presentations.}}\label{sec333}

\begin{parrafo}
Fix a transversal smooth projection $\beta:V^{(d)}\longrightarrow V^{(d-1)}$ (Definition \ref{trpt}) and a simple $\beta$-differential algebra $\G$. In Definition \ref{pr_locp} we introduced the notion of $p$-presentation, say $p\P=p\P(\beta,z,f_{p^e}(z))$. The aim of this Section is to define two functions:
\begin{enumerate}
\item a function $Sl(p\P)(-):V^{(d-1)}\longrightarrow \mathbb{Q}$ (Definition \ref{defslp}).
\item a function $\beta-ord^{(d-1)}(\G)(-):V^{(d-1)}\longrightarrow \mathbb{Q}$ (Definition \ref{def:z:adap:x}).
\end{enumerate}

There are many $p$-presentations $p\P$ which make use of the fixed projection $\beta$. Each $p$-presentation will define a function $Sl(p\P)$. The value of the new function $\beta-ord^{(d-1)}(\G)$ at a given point $y\in V^{(d-1)}$ will be given by the biggest value of the form $Sl(p\P)(y)$ among all $p$-presentations making use of $\beta$.

Over fields of characteristic zero,  the function $\beta-ord^{(d-1)}(\G)$ coincides  with the upper-semicontinous function $\ord(\R_{\G,\beta})$ (see \ref{dford}).  The situation in positive characteristic is quite different, for example  $\beta-\ord^{(d-1)}(\G)$ is not upper-semi-continous. Theorem \ref{rmk:x:restr} features a peculiar behavior of the function $Sl(p\P)$, which also leads to a simplification which will be crucial in our further development.

The function in (2) is a first step in the definition of our inductive function $v-ord^{(d-1)}$ in Section \ref{secMT}. 
In this section we simple fix a transversal projection $\beta$ and study different rational numbers, attached to a point,  defined by choosing different transversal sections $z=0$. We focus here, essentially, on how the function in (1) varies for different choices of $z$.
\end{parrafo}

\begin{definition}\label{defslp}
 Fix $\G$, $\beta:V^{(d)}\longrightarrow V^{(d-1)}$,  a  $\beta$-section $z$, and $f_{p^e}(z)$ as in \ref{pr_locp}. Namely,  fix a $p$-presentation $p\P(\beta, z,  f_{p^e}(z))$ with 
$  f_{p^e}(z)=z^{p^{e}}+a_1^{}z^{p^{e}-1}+\dots+a_{p^{e}}^{}$
 as in (\ref{eqld}), so that $\G$ has the same integral closure as
$$\calo_{V^{(d)}}[f_{p^e}(z)W^{p^e},\Delta^{(\alpha)}(f_{p^e}(z))W^{p^e-\alpha}]_{1\leq\alpha\leq p^e-1}\odot\R_{\G,\beta}.$$
Define $Sl(p\P)(-)\!: V^{(d-1)}\longrightarrow\mathbb Q,$
$$Sl(p\P)(y):=\min_{1\leq j\leq p^e}\Big\{\frac{\nu_y(a_j)}{j},\ord(\R_{\G,\beta})(y)\Big\} ,$$
called the \emph{slope of $\G$ relative to $p\P=p\P(\beta, z,  f_{p^e}(z))$ at $y\in V^{(d-1)}$}.
\end{definition}

\begin{remark}\label{rk32}
 The function $\ord(\R_{\G,\beta})(-): V^{(d-1)}\longrightarrow \mathbb{Q}$ takes values with denominators in $\frac{1}{n}\mathbb{Z}$, for some integer $n>0$. Thus the same holds for the slope function: it takes values in $\frac{1}{n(p^e!)}\mathbb{Z}$. Moreover, both functions  take only finitely many values.
\end{remark}

\begin{remark}\label{rkcrit} 
Given a $p$-presentation $p\P=p\P(\beta,z,f_{p^e}(z))$, other $p$-presentations can be defined, for example by  changing the section $z$. Here, given $x_0\in V^{(d)}$  we study conditions on $z$ for which $Sl(p\P)(\beta(x_0))>0$.  Note here that we do not assume that $z$ vanishes at $x_0$.

The element $z$ in the $p$-presentation in Definition \ref{defslp} defines a closed set, say $\overline{V}=\{z=0\}\subset V^{(d)}$, which is a section of $\beta$. In particular, any point $y\in V^{(d-1)}$ can be identified with a point in $\overline{V}$, say $x\in\overline{V}\subset V^{(d)}$, namely $x=\beta^{-1}(y)\cap \{z=0\}$.

The value of the function $Sl(p\P)$ at a point $y\in V^{(d-1)}$ provides information of $\G$, locally at $x$.

In this Remark we show that $Sl(p\P)(y)>0$ if and only if $\ord(\G)(x)>0$.
Here, we discuss some equivalent formulations of this condition. 

Let $k(y)$ be the residue field of the local ring $\calo_{{V^{(d-1)}},y}$ and let $Z$ denote the restriction of $z$ to $\beta^{-1}(y)$, the fiber over
 $y$. Fix, as above, $x=\beta^{-1}(y)\cap\{z=0\}$. So $x$ is the unique point dominating $y$ for which $z$ is a non-invertible element at $\calo_{{V^{(d)}},x}$, and $z$ is of order one at this local ring.
 
 Here, $f_{p^e}(z)$ is a global function of $V^{(d)}$, and the restriction
 $$\overline{\beta}: V(\langle f_{p^e}(z)\rangle) \longrightarrow V^{(d-1)}$$ is finite and flat.
 On the other hand, $V(\langle f_{p^e}(z)\rangle)\cap \beta^{-1}(y)$ can be identified with the subscheme in $\Spec(k(y)[Z])$ defined by the monic polynomial, say 
 $$\overline{f_{p^e}(z)}=Z^{p^e}+\overline{a}_1Z^{p^e-1}+\dots+\overline{a}_{p^e}\in k(y)[Z].$$
 Moreover, $x\in V(\langle f_{p^e}(z)\rangle)$ if and only if $\overline{a}_{p^e}=0$. In fact $Z$ is the class of $z$ on the fiber.
 
The following are equivalent conditions for the section $z$ and the point $y$:
\begin{enumerate}
\item[(a)] $Sl(p\P)(y)> 0$.

\item[(b)]  $\overline{f_{p^e}(z)}=Z^{p^e}$ and 
$\ord(\R_{\G,\beta})(y)>0$.

\item[(c)] $\ord(\R_{\G,\beta})(y)>0$,  the induced finite map, 
$$\overline{\beta}: V(\langle f_{p^e}(z)\rangle) \longrightarrow V^{(d-1)},$$
has a unique point, say $x$, dominating $y$, and $z$ is a non-invertible at $\calo_{V^{(d)},x}$. 
\item[(d)] $\ord(\R_{\G,\beta})(y)>0$, $V(\langle f_{p^e}(z)\rangle)\cap \beta^{-1}(y)$ is a unique point, say $x$, the local rings $\calo_{V^{(d)},x}$ and $\calo_{V^{(d-1)},y}$ have the same residue field, say $k(x)=k(y)$,
 and if $\{y_1, \dots ,y_{s}\}$ is a regular system of parameters in $\calo_{V^{(d-1)},y}$, then  $\{y_1, \dots ,y_{s},z\}$ is a regular system of parameters in $\calo_{V^{(d)},x}$.

\end{enumerate}
\end{remark}

\begin{remark}\label{rmk45}

\ \  1) We shall prove in Proposition \ref{prop36} i) that if $x\in \Sing(\G)$ (i.e., if $\ord(\G)(x)\geq 1$), all conditions in d) will hold at $y=\beta(x)$, whenever  a $\beta$-transversal section $z$ is chosen with 
order one at  $\calo_{V^{(d)},x}$. Moreover, in such case  $Sl(p\P)(y)\geq 1$.

2) To prove that (c) implies (d) note that if (c) holds, then $Z$ divides  $\overline{f_{p^e}(z)}$. As there is only one factor, then
$\overline{f_{p^e}(z)}=Z^{p^e}$, and so the point $x$ must be rational over $k(y)$.
\end{remark}

\begin{theorem}\label{rmk:x:restr}
Fix  $\G$ and $p\P=p\P(\beta, z,  f_{p^e}(z))$ as in {\rm \ref{pr_locp}}.
If $Sl(p\P)(y)=\displaystyle\frac{\nu_y(a_j)}{j}$ for some index $j\in\{1,\dots,p^e-1\}$, then $Sl(p\P)(y)=\ord(\R_{\G,\beta})(y)$. In particular,
$$Sl(p\P)(y)=\min\Big\{\frac{\nu_y(a_{p^e})}{p^e},\ord(\R_{\G,\beta})(y)\Big\}.$$
\end{theorem}

\begin{proof}
Let $n\in\{1,\dots,p^e-1\}$ be the smallest index for which $Sl(p\P)(y)=\displaystyle\frac{\nu_y(a_n)}{n}$. That is,
\begin{equation}\label{eq:ineq:ord}
\frac{\nu_y(a_n)}{n}<\frac{\nu_y(a_i)}{i}\ \hbox{ for }\ i\leq n-1\ \ \ \hbox{ and }\ \ \ \frac{\nu_y(a_n)}{n}\leq \frac{\nu_y(a_\ell)}{\ell}\ \hbox{ for }\ \ell\geq n+1.
\end{equation}
Recall the definition of the $\beta$-differential operators $\Delta^{(r)}$ in $\ref{pq}$. As $\G$ is assumed to be a $\beta$-differential algebra, then
$\Delta^{(p^e-n)}(f_{p^e}(z))W^{n}\in \G$.

Note that 
$$\Delta^{(p^e-n)}(f_{p^e}(z))W^{n}=(c_1a_1z^{n-1}+\dots+c_{n-1}a_{n-1}z+a_{n})W^{n}\in\G$$ 
for some elements $c_i\in k$ for $i=1,\dots,n-1$.

Let $\overline{\Delta^{p^e-n}(f_{p^e}(z))}W^n$ denote the  class of $\Delta^{(p^e-n)}(f_{p^e}(z))W^{n}$ in $\calo_{V^{(d)}}/\langle f_{p^e}(z) \rangle[W]$. The scheme $\calo_{V^{(d)}}/\langle f_{p^e}(z) \rangle[W]$ is a finite and free extension of $\calo_{V^{(d-1)}}[W]$.
The norm of the element
$$\overline{\Delta^{p^e-n}(f_{p^e}(z))}W^n=(c_1a_1\overline{z}^{n-1}+\dots+c_{n-1}a_n\overline{z}+a_{n})W^{n}$$ 
 over $\calo_{V^{(d-1)}}[W]$ is an element of the elimination algebra of $f_{p^e}(z)$, and hence of
  $\R_{\G,\beta}$ (see \cite{VV4}). Denote this element by $G(a_1,\dots,a_{p^e})W^{t}\in\R_{\G,\beta}$. In addition, in this case $t=np^e$, and $G(V_1,\dots,V_{p^e})\in k[V_1, \dots ,V_{p^e}]$ is a weighted homogeneous of degree $t=p^e$ provided each $V_i$ is given weight $i$.

Note that,
\begin{enumerate}
\item $G(a_1,\dots,a_{p^e})=a_n^{p^e}+\widetilde{G}(a_1,\dots,a_{p^e})$.
\item $\widetilde{G}(a_1,\dots,a_{p^e})\in\langle a_1,\dots,a_{n-1}\rangle$.
\end{enumerate}
To check the last assertion set formally $a_1=0,\dots,a_{n-1}=0$, in which case  $\overline{\Delta^{p^e-n}(f_{p^e}(z))}W^n=a_nW^n$, which has norm $a_n^{p^e}W^{n{p^e}}$. 

 Here $\widetilde{G}$ is a weighted homogeneous polynomial of degree $np^e$, and each monomial in $\widetilde{G}$ is of the form $a_1^{\alpha_1}\dots a_{p^e}^{\alpha_{p^e}}$ with $\sum_{j=1}^{p^e}j\alpha_j=np^e$, and $\alpha_j\not=0$ for some $j<n$ (as $\widetilde{G}(a_1,\dots,a_{p^e})\in\langle a_1,\dots,a_{n-1}\rangle$).

We claim that $\nu_y(a_1^{\alpha_1}\dots a_{p^e}^{\alpha_{p^e}})>\nu_y(a_n^{p^e})=p^e\nu_y(a_n)$ for any monomial in $\widetilde{G}$. In fact:
$$\nu_y(a_1^{\alpha_1}\dots a_{p^e}^{\alpha_{p^e}})=\sum_{j=1}^{p^e}\alpha_j\nu_y(a_j)>\sum_{j=1}^{p^e}\alpha_j j\frac{\nu_y(a_n)}{n}=np^e\frac{\nu_y(a_n)}{n}=\nu_y({a_n^{p^e}}),$$
where the inequality follows from the hypotheses in (\ref{eq:ineq:ord}). In particular, $\nu_y(\widetilde{G})>\nu_y(a_n^{p^e})$.

This proves that the order of $GW^{np^e} (\in\R_{\G,\beta})$ is $\frac{\nu_{y}(G)}{np^e}=\frac{\nu_y(a_n^{p^e})}{np^e}=\frac{\nu_y(a_n)}{n}.$
Hence $\ord(\R_{\G,\beta})(y)\leq\frac{\nu_y(G)}{np^e}=\frac{\nu_y(a_n)}{n}=Sl(p\P)(y)$. Finally, this inequality together with $Sl(p\P)(y)\leq\ord(\R_{\G,\beta})(y)$ implies that $Sl(p\P)(y)=\ord(\R_{\G,\beta})(y)$.
\end{proof}

\begin{remark}\label{rmk:Slxy}
Let $p\P$ be a $p$-presentation defined in a neighborhood of a closed point $x\in V^{(d-1)}$ and assume $x\in\overline{y}$ for some $y\in V^{(d-1)}$. In this case,
$$Sl(p\P)(y)\leq Sl(p\P)(x).$$
Recall that
$Sl(p\P)(y)=\big\{\frac{\nu_y(a_{p^e})}{p^e},\ord(\R_{\G,\beta})(y)\big\}.$
Since $p\P$ is defined in a neighborhood of $x$, then $\nu_y(a_{p^e})\leq\nu_x(a_{p^e})$. The upper-semicontinuity of $\ord(\R_{\G,\beta})$ implies that $\ord(\R_{\G,\beta})(y)\leq \ord(\R_{\G,\beta})(x)$. Thus $Sl(p\P)(y)\leq Sl(p\P)(x)$.
\end{remark}

\begin{proposition}\label{prop36}
Fix  $\G$ and $\beta: V^{(d)}\longrightarrow V^{(d-1)}$ together with a $p$-presentation $p\P=p\P(\beta, z,  f_{p^e}(z))$ as in {\rm \ref{pr_locp}}.

\begin{enumerate}
\item[i)] Suppose that $Sl(p\P)(y)>0$ at $y\in V^{(d-1)}$ and let  $x$ be the unique point in $V(\langle f_{p^e}\rangle)$ mapping to $y$ (see Remark {\rm\ref{rkcrit}}). Then,
 $$x\in \Sing(\G)\ \hbox{ if and only if }\ Sl(p\P)(y)\geq 1.$$

\item[ii)] If $\beta^{-1}(y) \cap \Sing(\G) \neq \emptyset $, then  $\beta^{-1}(y) \cap \Sing(\G) $ is a unique point, say $q$, and:

\begin{enumerate}
\item[iia)] If $Sl(p\P)(y)>0$, then $q$ is the unique point in $V(\langle f_{p^e}\rangle)$ that maps to $y$.

\item[iib)] If $Sl(p\P)(y)=0$, then the class of $a_{p^e}$ is a $p^e$-th power in 
$k(y)$, say $\overline{a}_{p^e}=\alpha^{p^e}$, and the class of $a_{i}$ is zero for $i=1, \dots , p^e-1$. Namely,
$$\overline{f_{p^e}(z)}=Z^{p^e}+\alpha^{p^e}\in k(y)[Z].$$

\end{enumerate}
\end{enumerate}
\end{proposition}

\begin{proof}\ \ 
i) Fix a regular system of parameters $\{y_1, \dots , y_s\}$ in $\calo_{V^{(d-1)},y}$. In this case, $\{y_1, \dots , y_s,z\}$ is a regular system of parameters in $\calo_{V^{(d)},x}$, so 
$f_{p^e}(z)=z^{p^{e}}+a_1^{}z^{p^{e}-1}+\dots+a_{p^{e}}^{}\in m^{p^e}_{x}$ if and only if $a_i \in m_{y}^{i}$. The equivalence now follows straightforward.

\vspace{0.15cm}

ii) Note that $\Sing(\G)\subset V(\langle f_{p^e}\rangle)$. Moreover,  
$\Sing(\G)\subset \mathcal{F}_{p^e}$, the closed set of points  of multiplicity $p^e$ of the hypersurface $V(\langle f_{p^e}\rangle)$. A theorem of Zariski states that $\beta$ induces a set theoretical bijection: $\beta: \mathcal{F}_{p^e} \longrightarrow \beta(\mathcal{F}_{p^e})$, and matching points have the same residue field. This proves property P0) in \ref{csc} (see \cite{BV3}, 8.4). In particular $\beta^{-1}(y) \cap \Sing(\G) $ is a unique point.

iia) As $q \in  V(\langle f_{p^e}\rangle)$, the assertion follows from the equivalence of a) and c) in \ref{rkcrit}.

\vspace{0.15cm}

iib) In this case, $y=\beta(q)\in \Sing(\R_{\G,\beta})$ (see P3) in \ref{csc}), so $\ord(\R_{\G,\beta})(y)\geq 1 $. On the other hand,  as $q \in \mathcal{F}_{p^e} $, then $k(q)=k(y)$. This together with 
Theorem \ref{rmk:x:restr} imply that 
$$\overline{f_{p^e}(z)}=Z^{p^e}+\overline{a}_{p^e}\in k(y)[Z],$$
in \ref{rkcrit}, and that this purely inseparable polynomial is a $p^e$-th power of a monic polynomial of degree 1, say
$Z^{p^e}+\overline{a}_{p^e}=(Z+\alpha)^{p^e}$ in $ k(y)[Z].$
\end{proof}

\begin{corollary}\label{corol:2pres:geq1}
  Fix two $p$-presentations for $\G$ on 
$V^{(d)}$. Say, $p\P$, defined in terms of
$\beta:V^{(d)}\longrightarrow V^{(d-1)}$, a $\beta$-section $z$ and a monic polynomial $f_{p^e}(z)$; and another $p$-presentation $p\P'$ defined by $\beta':V^{(d)}\longrightarrow V'^{(d-1)}$, a $\beta'$-section $z'$, and a  polynomial $f'_{p^e}(z')$.

Fix points $y\in V^{(d-1)}$, $y'\in V'^{(d-1)}$, and assume that:

\begin{enumerate}
\item[1)] $Sl(p\P)(y)>0$ and $Sl(p\P')(y')>0$.

\item[2)] There is a point $q\in V^{(d)}$ which is the unique point mapping to both (see \ref{rkcrit}). Namely, $\beta(q)=y$ and  $\beta'(q)=y'$.

\end{enumerate}

Then,
$Sl(p\P)(y)\geq 1$ if and only if $Sl(p\P')(y)\geq 1.$
In fact, this condition holds when both $y$ and $y'$ are image of a point $q\in \Sing(\G)$.
\end{corollary}

\begin{parrafo} \label{rdld}
In what follows we fix the simple algebra $\G$ on a smooth scheme $V^{(d)}$, together with a transversal morphism $\beta:V^{(d)}\longrightarrow V^{(d-1)}$, and define different $p$-presentations of the form $p\P(\beta,  z,  f_{p^e}(z))$,
(for different choices of sections $z$).

Let us denote by $\mathcal{F}(\G, \beta)$ the set of all such $p$-presentations. Namely,
$$\mathcal{F}(\G, \beta)=\big\{ p\P(\beta,  z,  f_{p^e}(z))\ \hbox{for which (\ref{eqdpl}) holds}\big\}$$

There is a natural notion of restriction on local presentations. Let 
${\U}^{(d-1)}$ be an open subset in $V^{(d-1)}$, and set ${\U}^{(d)}$ as the inverse image of ${\U}^{(d-1)}$. There is a natural restriction of $\G$ , say $\G|_{\U}$, of $\beta$, say ${\beta}|_\U:{\U}^{(d)}\longrightarrow {\U}^{(d-1)}$,
and of the $p$-presentation $p\P$, so that (\ref{eqdpl}) holds at the restriction. For each open ${\U}^{(d-1)}\subset {V}^{(d-1)}$, we take all $p$-presentations $\mathcal{F}(\G|_\U, \beta|_\U)$. 

Finally, fix a point $y\in {V}^{(d-1)}$, and set
$$\mathcal{F}(\G, \beta, y)=\bigcup \mathcal{F}(\G|_\U, \beta|_\U),$$ 
where the union is over all restrictions ${\U}^{(d-1)}\subset {V}^{(d-1)}$ containing $y$.

\end{parrafo}

\begin{definition}\label{def:z:adap:x} 
Fix $\beta:V^{(d)}\longrightarrow V^{(d-1)}$ and $\G$ as in \ref{pq}.
Define the \emph{$\beta$-order at } $y\in V^{(d-1)}$ as
\begin{equation}\label{eq:def:v_ord} 
\beta-ord^{(d-1)}(\G)(y)=\!\!\!\!\!\!\max_{p\P\in \mathcal{F}(\G, \beta, y) }\!\!\!\!\big\{ Sl(p\P)(y) \}.
\end{equation}

\end{definition}

\section{Well-adapted $p$-presentations.}\label{sec555}

\begin{parrafo}\label{rpcasos}Assume that $\beta-ord^{(d-1)}(\G)(y)>0$, and let $p\P$ be a $p$-presentation involving $\beta$. Here we sketch a criteria which will allow us to decide when, for a given point $y\in V^{(d-1)}$, a $p$-presentation $p\P$ is such that $\beta-ord^{(d-1)}(\G)(y)=Sl(p\P)(y)$.  So well-adapted $p$-presentations at a singular point  will ultimately be giving us the value of the inductive function at such point (see Corollary \ref{def_vord}).

The starting point of this discussion grows from the observation that when $Sl(p\P)(y)>0$, the following cases can occur:
\begin{enumerate}
\item[A)] $Sl(p\P)(y)=\ord(\R_{\G,\beta})(y)$

\item[B)] $Sl(p\P)(y)=\frac{\nu_y(a_{p^e})}{p^e}<\ord(\R_{\G,\beta})(y)$ (see Theorem \ref{rmk:x:restr}), and

\begin{enumerate}
\item[B1)]  $ \frac{\nu_y(a_{p^e})}{p^e} \notin \mathbb Z_{>0}$.

\item[B2)]  $ \frac{\nu_y(a_{p^e})}{p^e} \in \mathbb Z_{>0}$ and $\In_y(a_{p^e})$ is not a $p^e$-th power at $\Gr_y(\calo_{V^{(d-1)},y})$.

\item[B3)]  $ \frac{\nu_y(a_{p^e})}{p^e} \in \mathbb Z_{>0}$ and $\In_y(a_{p^e})$ is  a $p^e$-th power at $\Gr_y(\calo_{V^{(d-1)},y})$.
\end{enumerate}

\end{enumerate}
We shall prove that a new $p$-presentation $p\P'$ can be defined with the condition $Sl(p\P')(y)>Sl(p\P)(y)$, only in case B3). This leads to the \emph{cleaning process} developed in Proposition \ref{prpri}. 

This cleaning process relies on suitable changes of the transversal section $z$. The finiteness of this process will be address in Remark \ref{rk211}. In Proposition \ref{adap:mon}  we show that these changes of $z$, in this cleaning process, can be done so as to be compatible with the notion of monomial contact; a property that will be used in the proof of Main Theorem 2.

Proposition \ref{simult:adap} will be useful in the study of $p$-presentations and its compatibility with monomial transformations.
\end{parrafo}

\begin{parrafo}\label{cpll}

Let $p\P=p\P(\beta,z,f_{p^e}(z))$  be a $p$-presentation and fix $y\in V^{(d-1)}$. Suppose $Sl(p\P)(y)>0$.
We study changes of the $p$-presentation $p\P$ obtained by changing the $\beta$-section $z$  by another of the form $ uz+\alpha$. Here $u$ and $\alpha$ are in $\calo_{V^{(d-1)},y}$ and  $u$ is a unit, so the change is a composition of $z_1=u z$ and $z_2=z+\alpha$. The function $u$ is a unit (invertible) at any point in an open neighborhood of $y$, say $ {\U}^{(d-1)}$. This is to be interpreted as a new $p$-presentation, defined  at the restriction of both $\G$ and $ V^{(d)}$ over $ \U^{(d)}=\beta^{-1}(\U^{(d-1)})$ as in \ref{rdld}.

For a change of the form $z_1=u z$,  set  $p\P_1$ with
$$f'_{p^e}(z_1)=u^{p^e}f_{p^e}(z)=z_1^{p^e}+ua_1z_1^{p^e-1}+\dots+u^{p^e}a_{p^e}\in  \calo_{V^{(d-1)},y}[z_1] . $$

Clearly, $Sl(p\P)(y)=Sl(p\P_1)(y)$ and also Cases A), B1), B2), and B3) in \ref{rpcasos} are preserved. Henceforth we study only changes of the form $z^{\prime}=z+\alpha$.

 At $\calo_{V^{(d-1)},y}[z] =\calo_{V^{(d-1)},y}[z^{\prime}]$,
\begin{equation}\label{1ecfin}
f_{p^e}(z)=f^{\prime}_{p^e}(z^{\prime})={z^{\prime}}^{p^e}+a^{\prime}_1{z^{\prime}}^{p^e-1}+\dots+a^{\prime}_{p^e}\in  \calo_{V^{(d-1)},y}[z^{\prime}] , \mbox{ and }
\end{equation} 
\begin{equation}\label{ecfin}
a^{\prime}_{p^e}= \alpha^{p^e}+a_1\alpha^{p^e-1}+\dots+a_{p^e}.\
\end{equation}
Define, as before, a new presentation, say $p\P'$, with these data at a suitable restriction to a neighborhood of $y$.
\end{parrafo}

\begin{proposition} \label{prpri} {\bf (Cleaning process)}. Fix the setting and notation as above, where the function $\beta-ord^{(d-1)}(\G)(y) >0$ and where $p\P$ is such that $Sl(p\P)(y)>0$.

Assume that $Sl(p\P)(y)=\displaystyle \frac{\nu_y(a_{p^e})}{p^e}< \ord(\R_{\G,\beta})(y)$. There will be a change of the form $z'= z+ \alpha$, defining a new presentation $p\P'$ as in \emph{\ref{cpll}}, so that 
$$Sl(p\P)(y)<Sl(p\P')(y)\ \hbox{ if and only if case \emph{B3)} holds in {\rm\ref{rpcasos}} for }p\P.$$
\end{proposition}
\begin{proof}Theorem \ref{rmk:x:restr} ensures that if  $Sl(p\P)(y)=\displaystyle\frac{\nu_y(a_{p^e})}{p^e}< \ord(\R_{\G,\beta})(y)$, then
\begin{equation}
\frac{\nu_y(a_{p^e})}{p^e}< \frac{\nu_y(a_{i})}{i} \mbox{ for } i=1, \dots , p^e-1.
\end{equation}
Set $z'=z+\alpha$ as above. If $\nu_y(\alpha)< \frac{\nu_y(a_{p^e})}{p^e}$ then the previous inequalities applied to (\ref{ecfin}) 
show that $\nu_y(a^{\prime}_{p^e})=\nu_y(\alpha^{p^e}),$
so $Sl(p\P')(y)<Sl(p\P)(y)$. 

Assume that $\nu_y(\alpha)\geq  \frac{\nu_y(a_{p^e})}{p^e}.$
For each summand in (\ref{ecfin}) of the form $a_i\alpha^{p^e-i}$, $i=1, \dots , p^e-1$,
\begin{align}
\begin{aligned}
\nu_y(a_i\alpha^{p^e-i})&=(p^e-i)\nu_y(\alpha)+\nu_y(a_i)>\\
&> (p^e-i)\nu_y(\alpha)+ i  \frac{\nu_y(a_{p^e})}{p^e}
 \geq (p^e-i) \frac{\nu_y(a_{p^e})}{p^e} + i \frac{\nu_y(a_{p^e})}{p^e} =\nu_y(a_{p^e}).
\end{aligned}
\end{align}
Therefore (\ref{ecfin}) can be expressed as
\begin{equation}\label{nls}
 a^{\prime}_{p^e}= \alpha^{p^e}+A+a_{p^e},
 \end{equation}
where $\nu_y(\alpha^{p^e})\geq \nu_y(a_{p^e})$ and $\nu_y(A)>\nu_y(a_{p^e})$.

On the other hand, 
$$a'_n=\Delta^{(p^e-n)}(f_{p^e})(\alpha)=c_1\alpha^{n-1}a_1+\dots+c_{n-1}\alpha a_{n-1}+a_n,$$
where $c_j\in k$ for $j=1,\dots,n-1$. 

For each summand of the form $a_j\alpha^{n-j}$, $j=1, \dots , n$,
\begin{align}
\begin{aligned}
\nu_y(a_j\alpha^{n-j})&=(n-j)\nu_y(\alpha)+\nu_y(a_j)>\\
&>(n-j)\nu_y(\alpha)+ j  \frac{\nu_y(a_{p^e})}{p^e}
 \geq (n-j) \frac{\nu_y(a_{p^e})}{p^e} + j \frac{\nu_y(a_{p^e})}{p^e} =\frac{n \nu_y(a_{p^e})}{p^e}.
\end{aligned}
\end{align}
In particular, $\frac{\nu_y(a'_n)}{n}>\frac{\nu_y(a_{p^e})}{p^e}.$

One can easily check now that if B1) holds, then $f^{\prime}_{p^e}(z')={z^{\prime}}^{p^e}+a^{\prime}_1{z^{\prime}}^{p^e-1}+\dots+a^{\prime}_{p^e}$ in 
(\ref{1ecfin}) is also in case B1), and $Sl(p\P)(y)=Sl(p\P')(y)$. 

The same arguments apply if B2) holds, namely $f^{\prime}_{p^e}(z')$ is also in case B2), and $Sl(p\P)(y)=Sl(p\P')(y)$.

On the contrary, in case B3) it suffices to choose $\alpha$ so that 
$\nu_y(\alpha^{p^e}+a_{p^e})> \nu_y(a_{p^e})$ to get $Sl(p\P)(y)<Sl(p\P')(y).$
\end{proof}

\begin{remark}\label{rmk44pe} In Proposition \ref{prpri} we assumed that $\beta-ord^{(d-1)}(\G)(y)>0$, and $p\P$ was such that $Sl(p\P)(y)>0$. 
Suppose that $\beta-ord(\G)(y)>0$ and $Sl(p\P)(y)=0$. 
In this case, $\nu_y( a_{p^e})=0$ and
$$\overline{f_{p^e}(z)}=Z^{p^e}+\overline{a}_{p^e} \in k(y)[Z].$$
Namely, $\overline{a}_i=0$, $i=1, \dots , p^e-1$ (see Theorem \ref{rmk:x:restr}).

If $Z^{p^e}+\overline{a}_{p^e}= (Z+ \delta)^{p^e} \in k(y)[Z]$
for some $\delta \in k(y)$, then a change of the form $z^{\prime}=z+\alpha$, where $\alpha \in \calo_{V^{(d-1)}, y}$ maps to $\delta$ in $k(y)$, will define a new presentation, say $p\P'$, and $Sl(p\P')(y)>0$. This is always the case when $y$ is the image of a point $x\in \Sing(\G)$ (see Proposition \ref{prop36} iib)).
\end{remark}

\begin{definition}\label{defweladap}
Let $p\P$ be a $p$-presentation. We say that $p\P$ is \emph{well-adapted to $\G$ at $y\in V^{(d-1)}$} whenever one of the two cases holds:
\begin{enumerate}
\item[(1)] $Sl(p\P)(y)>0$ and either case A), case B1) or case B2) in \ref{rpcasos} holds,

\item[(2)]  $Sl(p\P)(y)=0$ and 

\begin{enumerate}
\item[A)] $\ord(\R_{\G,\beta})(y)=0$, or
\item[B2)] $\nu_y(a_{p^e})=0<\ord(\R_{\G,\beta})(y)$ and $\overline{a}_{p^e}\in k(y)$ is not a $p^e$-th power.
\end{enumerate}
\end{enumerate}
\end{definition}

\begin{remark}\label{rmk56}
Fix $q\in\Sing(\G)$ and a $p$-presentation $p\P=p\P(\beta,z,f_{p^e}(z))$ which is well-adapted to $\beta(q)$, then
 $Sl(p\P)(\beta(q))\geq 1$ (Proposition \ref{prop36}), and 
 $z$ is an element of order one in $\calo_{V^{(d)},q}$ (Remark \ref{rmk45}).
\end{remark}

\begin{remark}\label{rk211} {\bf Finiteness of the cleaning process}.

When Case B3) occurs,  $\nu_y(a_{p^e})=\ell p^e$ for some integer $\ell\geq 1$, and $$\In_y(a_{p^e})= F^{p^e}$$ for some homogenous polynomial $F$ of degree $\ell$ at $\Gr_y(\calo_{V^{(d-1)}})$.
In this case, we define $z^{\prime}=z+\alpha$ for some
$\alpha \in \calo_{V^{(d-1)},y}$ such that $\In_y(\alpha)=F$. Thus 
$Sl(p\P)(y)<Sl(p\P')(y)$.

If this new presentation $p\P'=p\P'(\beta,z',f'_{p^e}(z'))$ is within Case A), B1) or B2) then stop. If, on the contrary, $f^{\prime}_{p^e}(z')$ is in case B3), then
\begin{enumerate}
\item[i)] $\nu_y(a^{\prime}_{p^e})=\ell^{\prime} p^e$ (with $\ell^{\prime}>\ell)$, and 

\item[ii)] $\In_y(a^{\prime}_{p^e})= (F^{\prime})^{p^e}$ for some homogeneous element $F^{\prime}$ of degree $\ell^{\prime}$ at $\Gr_y(\calo_{V^{(d-1)},y})$.
\end{enumerate}

So again we can set $z^{\prime \prime}=z^{\prime }+\alpha^{\prime}$ for some
$\alpha^{\prime } \in \calo_{V^{(d-1)},y}$ with $\In_y(\alpha^{\prime })=F^{\prime }$; and
$Sl(p\P^{\prime })(y)<Sl(p\P^{\prime \prime })(y)$. This shows that with this procedure of modification of the transversal section, locally over $y$, the slope will increase every time we come to Case B3). Finally, 
Remark \ref{rk32} guarantees that Case B3) can arise only finitely many times throughout this procedure. So ultimately the procedure leads to a well-adapted $p$-presentation.
\end{remark}

\begin{proposition}\label{adap:mon}
Assume that $p\P$ is compatible with a monomial algebras $\calo_{V^{(d)}}[\m W^s]$ as in \ref{defmon}. Then the cleaning process to obtain a well-adapted $p$-presentation at a point  preserves the compatibility with $\calo_{V^{(d)}}[\m W^s]$.
\end{proposition}

\begin{proposition}{\bf Simultaneous adaptation}\label{simult:adap}. 

Let $p\P=p\P(\beta,  z,  f_{p^e}(z))$
be a $p$-presentation compatible with a monomial algebra $\m W^s$. Let $y$ and $x$ be points in $V^{(d-1)}$, so that $x\in C= \overline{y}$, and assume that $\calo_{C,x}$ is regular. 
Then,
\begin{enumerate}
\item[A)] It can be assumed that $p\P$ is well-adapted to $\G$ at $y$, defined in a neighborhood of $x$, and compatible with $\m W^s$.

\item[B)] There is a $p$-presentation which is well-adapted to $\G$ both at $y$ and $x$, and also compatible with $\m W^s$. 
\end{enumerate}
\end{proposition}

\begin{proof2}
Once we fix a $p$-presentation, say $p\P=p\P(\beta,z,f_{p^e})$ and a point $y\in V^{(d-1)}$, cleaning applies either in the case of Remark \ref{rmk44pe} or of Remark \ref{rk211}. In both cases, the condition is given by the fact that $In_y(a_{p^e})$ is a $p^e$-th power, cleaning consists in finding $\alpha\in\calo_{V^{(d)},y}$ so that $\big(In_y(\alpha)\big)^{p^e}=In_y(a_{p^e})$.

We face now the proof of Proposition \ref{simult:adap}. Fix a $p$-presentation $p\P$ locally defined at $\calo_{V^{(d-1)},x}$. Set $\p\subset\calo_{V^{(d-1)},x}$ the regular prime ideal corresponding to $y$ (so that localization at $\p$ is $\calo_{V^{(d-1)},y}$). Cleaning is necessary at $\calo_{V^{(d-1)},y}$ if and only if $In_y(a_{p^e})$ is a $p^e$-th power in $gr_y(\calo_{V^{(d-1)},y})$. Since $x$ is a smooth point at $\overline{y}$, then $gr_\p(\calo_{V^{(d-1)}})$ is a regular ring and $In_\p(a_{p^e})\in gr_\p(\calo_{V^{(d-1)},x})$ and by localization by passing from $[gr_\p(\calo_{V^{(d-1)}})]_0$ to the total quotient field, we get $In_y(a_{p^e})\in gr_y(\calo_{V^{(d-1)}})$. Hence $In_\p(a_{p^e})$ is a $p^e$-th power if and only if $In_y(a_{p^e})$ is a $p^e$-th power. This ensures that the element $\alpha$, used in the cleaning process at $y$, can be chosen to be an element in $\calo_{V^{(d-1)},x}$, and hence the cleaning process at $y$ can be done so as to obtain a new $p$-presentation with coefficients in $\calo_{V^{(d-1)},x}$.

Once we assume that $p\P$ is well-adapted at $y$, we want to study the ``adaptability'' of $p\P$ at $x$. Note here that the only case to be considered occurs when $Sl(p\P)(x)=\frac{\nu_x(a_{p^e})}{p^e}<\ord(\R_{\G,\beta})(x)$ and $In_x(a_{p^e})$ is a $p^e$-th power. Here, we prove, under this last assumption, that the cleaning process at $x$ can be done without affecting the fact that the presentation is already well-adapted at $y$. 

Consider a regular system of parameters $\{y_1,\dots,y_\ell,y_{\ell+1},\dots,y_{d-1}\}$ at $\calo_{V^{(d-1)},x}$ so that $\p=\id{y_1,\dots,y_\ell}$. There are two cases to consider, case $Sl(p\P)(y)=\ord(\R_{\G,\beta})(y)$ and case $Sl(p\P)(y)=\frac{\nu_y(a_{p^e})}{p^e}<\ord(\R_{\G,\beta})(y)$. Assume that the latter case holds and set $\frac{\nu_y(a_{p^e})}{p^e}=\frac{n}{p^e}$. Theorem \ref{rmk:x:restr} says that $\frac{\nu_y(a_{p^e})}{p^e}<\frac{\nu_y(a_{j})}{j}$ for $j=1,\dots,p^e-1$.

At the completion, $a_{p^e}$ is a sum of monomials of the form $y_1^{\alpha_1}\dots y_\ell^{\alpha_\ell}y_{\ell+1}^{\alpha_{\ell+1}}\dots y_{d-1}^{\alpha_{d-1}}$ with $\alpha_1+\dots+\alpha_\ell\geq n$. We can identify $In_x(a_{p^e})$ with a sum of some of these terms. If there is an element $\alpha\in\calo_{V^{(d-1)},x}$ so that $\big(In_x(\alpha)\big)^{p^e}=-In_x(a_{p^e})$, then $\alpha\in\id{y_1,\dots,y_\ell}^{\lceil \frac{n}{p^e}\rceil}$ and, in particular, $\nu_y(\alpha)\geq \frac{n}{p^e}$

A change of variables of the form $z\mapsto z+\alpha$ produces a new independent coefficient of the form:
$$a'_{p^e}=\alpha^{p^e}+a_1\alpha^{p^e-1}+\dots+ a_{p^e},$$
where $\nu_y(a_j\alpha^{p^e-j})>j\cdot \frac{n}{p^e}+(p^e-j)\frac{n}{p^e}=n$, for $j=1,\dots,p^e-1$; so $\nu_y(a'_{p^e})\geq \min\{\nu_y(\alpha^{p^e}),\nu_y(a_{p^e}))\}\geq n$ and the new presentation is still well-adapted at $y$.

In case $Sl(p\P)(y)=\ord(\R_{\G,\beta})(y)=\frac{n}{s}$, the same arguments leads to the existence of $\alpha\in\calo_{V^{(d-1)}}$ so that $\nu_y(\alpha)=\frac{n}{s}$ which again ensures that the change of variables $z\mapsto z+\alpha$ (needed for the cleaning process at $x$) does not affect the slope at $y$.

To prove Proposition \ref{adap:mon} just notice that such change can be achieved with $\alpha W\in\calo_{V^{(d-1)}}[\m W^s]$. Similar arguments as before applies here to show that the new coefficients $a'_nW^n\in\calo_{V^{(d-1)}}[\m W^s]$ ($n=1,\dots,p^e$).
\end{proof2}

\section{{Transformations of $p$-presentations.}}\label{sec666} 

\begin{parrafo} In the previous sections some invariants were defined in terms of $p$-presentations. In this section we discuss a form of compatibility of these invariants when applying a  monoidal transformation along a smooth center $C$. 

The starting point will be a notion of transformation of $p$-presentations in \ref{tdpp}. A monoidal transformation defined by blowing-up a smooth center $C$ introduces an exceptional hypersurface, say $H$. The aim of the section is to relate the value of  the slope $Sl$ at the generic point of $H$  with the value of $Sl$ at the generic point of $C$ (see Proposition \ref{ppreOK}). This result will  be an essential ingredient for the proofs of Main Theorems in this work.
\end{parrafo}

\begin{parrafo}\label{tdpp}

Take a $p$-presentation $p\P=p\P(\beta,z,f_{p^e})$ of a simple $\beta$-differential algebra $\G$ on $V^{(d)}$. Namely, a smooth morphism $V^{(d)}\overset{\beta}{\longrightarrow}V^{(d-1)}$, a $\beta$-section $z$ and a monic polynomial $f_{p^e}(z)=z^{p^e}+a_1z^{p^e-1}+\dots+a_{p^e}$. Assume that $C\subset\Sing(\G)$ is a closed and smooth center, and that $z\in I(C)$. 
Locally at a closed point $x\in C$,  there is a regular system of parameters $\{z,x_1,\dots,x_{d-1}\}$ and, after restriction to a suitable neighborhood of $x$, $I(C)=\langle z,x_1,\dots,x_{\ell}\rangle$. 
Consider the commutative diagram 
\begin{equation}\label{eqcdfff}
\xymatrix@R=0pc@C=0pc{
\G   & &  &  &  &  \G_1\\
V^{(d)}\ar[ddd] ^\beta  & &  &  &  &  \quad V^{(d)}_{1}\ \ar[lllll]_{\pi_{C}}\ar[ddd]^{\beta_{1}}\\
\\
\\
V^{(d-1)} & &  &  & &  V^{(d-1)}_{1}\ar[lllll]_{\pi_{\beta(C)}}\\
\R_{\G,\beta}   & &  &  &  &  (\R_{\G,\beta})_1=\R_{\G_1,\beta_1}
}\end{equation}
and recall that $\Sing(\G_1) \subset V_1^{(d)}$ can be covered by affine charts $U_{x_i}$,
$$U_{x_i}=\Spec\Big(\calo_{V^{(d)}}\big[\frac{z}{x_i},\frac{x_1}{x_i},\dots,\frac{x_{i-1}}{x_i},x_i,\frac{x_{i+1}}{x_i},\dots,\frac{x_\ell}{x_i},x_{\ell+1}\dots,x_{d-1}\big]\Big),$$
for $i=1,\dots,\ell$; and also $V_1^{(d-1)}$  is covered by charts $U_{x_i}'$
$$U'_{x_i}=\Spec\Big(\calo_{V^{(d-1)}}\big[\frac{x_1}{x_i},\dots,\frac{x_{i-1}}{x_i},x_i,\frac{x_{i+1}}{x_i},\dots,\frac{x_\ell}{x_i},x_{\ell+1}\dots,x_{d-1}\big]\Big).$$
Note that the strict transform of $z$, say $z_1=\frac{z}{x_i}$, is a transversal parameter for $U_{x_i}\longrightarrow U'_{x_i}$. 

The hypersurface defined by $f_{p^e}$ at $V^{(d)}$ has multiplicity $p^e$ along points of $C$. Let
$$f_{p^e}^{(1)}(z_1)=z_1^{p^e}+a_1^{(1)}z^{p^e-1}+\dots+a_{p^e}^{(1)}$$
denote the strict transform of $f_{p^e}(z)$.
These data define, locally,  a $p$-presentation of $\G_1$, say $p\P_1=p\P_1(\beta_1,z_1,f_{p^e}^{(1)})$, which we call the \emph{transform} of $p\P=p\P(\beta,z,f_{p^e})$. 
\end{parrafo}

\begin{remark}\label{rmkzin}
\begin{enumerate}
\item In the previous discussion we have assumed that $z\in I(C)$. If $C$ is irreducible, this condition will hold for any $p$-presentation $p\P(\beta,z,f_{p^e})$ well-adapted at $\xi_{\beta(C)}$, (the generic point of 
$\beta(C)$ in $V^{(d-1)}$). In fact, after a suitable restriction to a neighborhood of the closed point $x\in C$, the simultaneous cleaning procedure  at $\beta(x)$ and $\xi_{\beta(C)}$, and the fact that $C\subset\Sing(\G)$ will allow us to modify $z$ so that $z\in I(C)$ (see Proposition \ref{prpri}). 
\item Note that the exponent $p^e$ (the degree of the monic polynomial), is also preserved by transformations of $p$-presentations.
\end{enumerate}
\end{remark}

\begin{remark}\label{rmk62same}
A point $y\in V_1^{(d-1)}$ has an image in $V^{(d-1)}$, say $\pi_{\beta(C)}(y)$. If $y$ is not in the exceptional locus of $\pi_{\beta(C)}$, there is an open neighborhood, say $U$, of $\pi_{\beta(C)}(y)$ over which both $\pi_C$ and $\pi_{\beta(C)}$ are the identity map. Thus the restriction of both $p$-presentations $p\P$  and $p\P_1$ to $U$ coincide. 

In particular, 
$$Sl(p\P_1)(y)=Sl(p\P)(\pi_{\beta(C)}(y))$$
whenever $y\in V_1^{(d-1)}$ is not on the exceptional locus. Moreover, if $p\P$ is well-adapted to $\G$ at $\pi_{\beta(C)}(y)$, then the same holds for $p\P_1$ at $y$. 
\end{remark}

\begin{remark}\label{lem:z1:sing}
Fix $x\in\Sing(\G)$ a closed point so that $\tau_{\G,x}= 1$ and assume that $C$ is a permissible center containing $x$. Let $p\P(\beta,z,f_{p^e})$ be a $p$-presentation. Denote by $y$ the generic point of $\beta(C)$. Assume that $p\P$ is well-adapted simultaneously at $\beta(x)$ and $y$ and, in particular, Remark \ref{rmkzin} (1) says that $z\in I(C)$.

The intersection of the strict transform of $f_{p^e}$ with the exceptional locus $\pi^{-1}_C(C)$ is defined by $In_C(f_{p^e})\subset gr_{I(C)}(\calo_{V^{(d)}} )$
As $C$ is an equimultiple center for $f_{p^e}$, the intersection of the strict transform with points of $\pi_C^{-1}(x)$ is determined by $In_x(f_{p^e})$.

Finally as $\tau_{\G,x}= 1$ and $p\P$ is well-adapted at $x$, then $In_x(f_{p^e})=Z^{p^e}$, and hence
  $x'\in \{z_1=0\}$ for any  $x'\in\Sing(\G_1)$ mapping to $x$, where   $z_1$ denotes  the strict transform of $z$.
\end{remark}

\begin{proposition}\label{ppreOK}
Let $C$ be a permissible center passing through a closed point $x\in\Sing(\G)$ and assume that $\tau_{\G,x}=1$. Fix a $p$-presentation $p\P(\beta,z,f_{p^e})$. Let $y$ denote the generic point of $\beta(C)$ and assume that $p\P$ is well-adapted to $\G$ both at $\beta(x)$ and at $y$. Define a monomial transformation with center $C$. Then:
\begin{enumerate}
\item The transform $p\P_1$ is well-adapted to $\G_1$ at $\xi_{H}$, (the generic point of the exceptional hypersurface $H\subset V^{(d-1)}$). Moreover,
$$Sl(p\P_1)(\xi_{H})=Sl(p\P)(y)-1.$$
\item If, in addition, $p\P$ is compatible with a monomial algebra, say $\calo_{V^{(d-1)}}[I(H_1)^{h_1}\dots I(H_r)^{h_r}W^s]$, then $p\P_1$ is compatible with the monomial algebra
$$\calo_{V^{(d-1)}}[I(H_1)^{h_1}\dots I(H_r)^{h_r}I(H)^{\gamma}W^s],$$
where $\frac{\gamma}{s}=Sl(p\P_1)(\xi_{H})=Sl(p\P)(y)-1$.
\end{enumerate}
\end{proposition}

\begin{proof} Note that (2) follows from (1).
Set
\begin{equation}
\xymatrix@R=0pc@C=0pc{
p\P   & &  &  &  &  p\P_1\\
V^{(d)}\ar[dddd] ^\beta  & &  &  &  &  \quad V^{(d)}_{1}\ \ar[lllll]_{\pi_{C}}\ar[dddd]^{\beta_{1}}\\
\\
 \\
\\
V^{(d-1)} & &  &  & &  V^{(d-1)}_{1}\ar[lllll]_{\pi_{\beta(C)}}\\
y=\xi_{\beta(C)}   & &  &  &  &  \xi_H
}\end{equation}
where $H$ is the exceptional hypersurface, and
$$f_{p^e}^{(1)}(z_1)=z_1^{p^e}+a_1^{(1)}z^{p^e-1}+\dots+a_{p^e}^{(1)}$$ 
is the strict transform of $f_{p^e}(z)$. At points of $ U_{x_i}$,  $z_1=\frac{z}{x_i}$ and the coefficients $a_n^{(1)}$ factor as 
\begin{equation}\label{eq:fact:an}
a_{n}^{(1)}=x_i^{\nu_y(a_n)-n}a'_{n}=x_i^{r_n}a_n',
\end{equation}
where $a'_{n}$ denotes the strict transform of $a_{n}$ and $r_n=\nu_y(a_n)-n$, for $n=1,\dots,p^e$.

Different cases can arise under these assumptions, we classify them as in \ref{rpcasos}: 
\begin{enumerate}
\item[(A)]
Suppose that  $Sl(p\P)(y)=\ord(\R_{\G,\beta})(y)$ and, in particular, that $\frac{\nu_y(a_j)}{j}\geq \ord(\R_{\G,\beta})(y)$. At the points of $\Sing(\G_1)\cap U_{x_i}$,
$$\frac{\nu_{\xi_{H}}(a_{p^e}^{(1)})}{p^e}=\frac{\nu_y(a_{p^e})}{p^e}-1\geq \ord(\R_{\G,\beta})(y)-1=\ord((\R_{\G,\beta})_1)(\xi_{H}).$$
Thus $p\P_1$ is well-adapted to $\G_1$ at $\xi_H$ (case A) in \ref{rpcasos} and Definition \ref{defweladap}).

\item[(B)] Suppose that
$$Sl(p\P)(y)=\frac{\nu_{y}(a_{p^e})}{p^e}<\ord(\R_{\G,\beta})(y).$$ 

\begin{enumerate}
\item[(B.1)] Assume now that  $\frac{\nu_{y}(a_{p^e})}{p^e}\not\in\mathbb{Z}_{>0}$.  In this case, $ \nu_{y}(a_{p^e})>p^e$ and, in addition, $\nu_{y}(a_j)>j$ for $j=1\dots,p^e-1$. In particular,
$\nu_{\beta(x)}(a_j)>j$ for $j=1,\dots,p^e$
 and hence $\In_{x}(f_{p^e})=Z^{p^e}$. 
 
Lemma \ref{lem:z1:sing} (2) applies so $\Sing(\G_1)\cap H_{i+1}\subset\{z_1=0\}$. Under these assumptions,
$$\frac{r_{p^e}}{p^e}=\frac{\nu_{y}(a_{p^e})}{p^e}-1<\frac{\nu_y(a_j)}{j}-1=\frac{r_j}{j}$$
for $j=1,\dots,p^e-1$, and $\frac{r_{p^e}}{p^e}\not\in\mathbb{Z}_{\geq0}.$
 So, locally at any closed point $x'\in\Sing(\G_1)$ mapping to $x$, $p\P_1=p\P_1(\beta_1,z_1,f_{p^e}^{(1)})$ is of the form B1) in \ref{rpcasos} and therefore well-adapted to $\G_1$ at $\xi_H$. 
 
\item[(B.2)]  Suppose that $\frac{\nu_{y}(a_{p^e})}{p^e}=r\in\mathbb{Z}_{>0}$ and $\tau_{\G,x}=1$. The singular locus at any exceptional point mapping to $x$ is contained in the strict transform of $z$ as indicated in Remark \ref{lem:z1:sing}. 

Consider $\In_{\beta(C)}(a_{p^e})\in \Gr_{I(\beta(C))}(\calo_{V^{(d-1)}})=\calo_{\beta(C)}[X_1,\dots,X_\ell]$, and set
$$\In_{\beta(C)}(a_{p^e})=\sum_{|\alpha|=r p^e}b_\alpha M^{\alpha},$$
which, by assumption, is not a $p^e$-th power.

\item[(B.2.a)]First assume that $\In_{\beta(C)}(a_{p^e})\not\in \calo_{\beta(C)}[X_1^{p^e},\dots,X_\ell^{p^e}]$.
In this case, as the degree of $\In_{\beta(C)}(a_{p^e})$ is a multiple of $p^e$, there is a multi-index $\alpha=(\alpha_1,\dots,\alpha_\ell)$ with at least two integers which are not multiple of $p^e$, and $b_\alpha\not=0$.

We claim now that $a'_{p^e}$ restricted to the exceptional hypersurface $x_i=0$, say $\overline{a_{p^e}'}$,  is not a $p^e$-th power. This can be checked using the existence of the previous
multi-index $\alpha=(\alpha_1,\dots,\alpha_\ell)$.

This ensures that $p\P_1$ is in the case B2) in \ref{rpcasos}, and hence that $p\P_1$ is well-adapted to $\G_1$ at $\xi_{H}$.

\item[(B.2.b)] Suppose now that any $M^\alpha$ is a $p^e$-th power whenever $|\alpha|=rp^e$. Recall that $In_{\beta(C)}(a_{p^e})$ is not a $p^e$-th power, so some $b_\alpha$ is not a $p^e$-th power. Setting as before $\overline{a_{p^e}'}$ as the restriction of $a_{p^e}'$ to the exceptional hypersurface, then one checks that $\overline{a_{p^e}'}$ is not a $p^e$-th power as $b_\alpha$ is not a $p^e$-th power.
So again $p\P_1$ is well-adapted at $\xi_{H}$.
\end{enumerate}
\end{enumerate}
\end{proof}


\section{On the two Main Theorems}\label{secMT}
\begin{parrafo} Fix a smooth scheme  $V^{(d)}$ and a simple algebra $\G$ which we assume to be an absolute differential algebra. This  ensures that $\G$ is a $\beta$-differential algebra for any smooth transversal morphism $\beta:V^{(d)}\longrightarrow V^{(d-1)}$. 
It is under  this last condition  that a function $\beta-ord^{(d-1)}(\G): V^{(d-1)}\longrightarrow \mathbb Q$ was defined in \ref{def:z:adap:x}.
The same holds for any other  $\beta^{\prime}:V^{(d)}\longrightarrow V'^{(d-1)}$ transversal to $\G$.

A sequence (\ref{unaseq}) of permissible transformations of $\G$
induces two diagrams
\begin{equation}\label{seq7102}
\xymatrix@R=0pc@C=-0.05pc{
\G & & & & & \G_1 &  & & & &  &  & & & &   \G_r &  & & &  & & & & \G & & & & & \G_1 &  & & & &  &  & & & &   \G_r\\
V^{(d)}\ar[ddd]^\beta  &  & & & &   V^{(d)}_{1}\ar[lllll]_{\pi_{C_1}}\ar[ddd]^{\beta_1}    & & & & & \dots \ar[lllll] &  & & & &   V^{(d)}_{r}\ar[lllll]_{\pi_{C_r}}\ar[ddd]^{\beta_r} &  & & & &  & & & V^{(d)}\ar[ddd]^{\beta^{\prime}}  &  & & & &   V^{(d)}_{1}\ar[lllll]_{\pi_{C_1}}\ar[ddd]^{\beta^{\prime}_1}    & & & & & \dots \ar[lllll] &  & & & &   V^{(d)}_{r}\ar[lllll]_{\pi_{C_r}}\ar[ddd]^{\beta^{\prime}_r}  \\
\\
\\
V^{(d-1)}  &  & & & &   V^{(d-1)}_{1}\ar[lllll]_{\pi_{\beta(C_1)}}   & & & & & \dots \ar[lllll] &  & & & &   V^{(d-1)}_{r}\ar[lllll]_{\pi_{\beta(C_r)}} &  & & &   & & & & {V^{\prime}}^{(d-1)}  &  & & & &   {V^{\prime}}^{(d-1)}_{1}\ar[lllll]_{\pi_{\beta'(C_1)}}   & & & & & \dots \ar[lllll] &  & & & &   {V^{\prime}}^{(d-1)}_{r}\ar[lllll]_{\pi_{\beta'(C_r)}}\\
\R_{\G,\beta} & & & & & (\R_{\G,\beta})_1 &  & & & &  &  & & & &   (\R_{\G,\beta})_r &  & & &  & & & & \R_{\G,\beta'} & & & & & (\R_{\G,\beta'})_1 &  & & & &  &  & & & &   (\R_{\G,\beta'})_r
}\end{equation}
%

%

\begin{theorem}\label{result43}{\rm (}{\bf  Main Theorem 1}{\rm )}.
Assume that the previous setting holds. Then, for any point $q\in \Sing(\G_r)$,
 $$\beta_r-ord(\G_r)(\beta(q))=\beta_r^{\prime}-ord(\G_r)(\beta'(q)).$$
 Moreover, if $p\P=p\P(\beta_r,z,f_{p^e})$ is well-adapted to $\G_r$ at $q$, then
 $$\beta_r-ord(\G_r)(\beta_r(q))=Sl(p\P)(\beta_r(q)).$$
\end{theorem}

\begin{corollary}\label{def_vord}
The previous result enables us to define a function along $\Sing(\G_r)$:
$$v-ord^{(d-1)}(\G_r)(-):\Sing(\G_r)\longrightarrow \mathbb{Q}.$$
Moreover, if $p\P=p\P(\beta_r,z,f_{p^e}=z^{p^e}+a_1z^{p^e-1}+\dots+a_{p^e})$ is well-adapted to $\G_r$ at $\beta_r(x)$ ($x\in\Sing(\G_r)$), then
$$v-ord^{(d-1)}(\G_r)(x)=\min\Big\{\frac{\nu_{\beta_r(x)}(a_{p^e})}{p^e},\ord(\R_{\G,\beta})(\beta_r(x))\Big\},$$

\end{corollary}

\end{parrafo}

\begin{parrafo}\label{asert} 
Recall that the exceptional locus of the composite map $V^{(d)}\overset{\pi}{\longleftarrow}V_r^{(d)}$ in (\ref{unaseq}), say $\{H_1,\dots,H_r\}$, is a set of hypersurfaces at $V_r^{(d)}$ and it is assumed that the union has only normal crossings.

We now attach to the sequence (\ref{unaseq}) a monomial algebra supported on the exceptional locus: 
\begin{equation}\label{eqmonomio}
\m_r W^s= \calo_{V_r^{(d)}}[I(H_1)^{h_1}\dots I(H_r)^{h_r}W^s],
\end{equation}
with exponents $h_i\in\mathbb{Z}_{\geq0}$ defined so that:
$$q_{H_i}:=\frac{h_i}{s}=v-ord^{(d-1)}(\G_{i-1})(\xi_{C_i})-1$$
where $\xi_{C_{i}}$ denotes the generic point of each center $C_i$ ($i=1,\dots,r$).
 
Here $s$ is a positive integer so that $\{q_{H_1},\dots,q_{H_r}\}\subset \frac{1}{s}\Z$. As Rees algebras are considered up to integral closure, $\m_r W^s$ is independent of the choice of $s$; and will be called the \emph{tight monomial algebra} of $\G_r$ or the \emph{tight monomial algebra} defined by (\ref{unaseq}). 
\end{parrafo}

\begin{theorem}{\rm (}{\bf Main Theorem 2}{\rm )}\label{MT2}.
Fix a sequence of permissible transformations as (\ref{unaseq}). Let $\m_r W^s$ denote the tight monomial algebra defined in \ref{asert}.
Then, at any closed point $x\in\Sing(\G_r)$,
$\m_r W^s$ has monomial contact with $\G_r$, i.e., there is a  $\beta_r$-transversal section $z$ of order one at  $\calo_{V^{(d)}_r,x}$ for which
$$\G_r\subset \id{z}W\odot \m_r W^s.$$
\end{theorem}

\begin{parrafo}The two previous Main Theorems will lead us to the notion of strong monomial case, to be discussed now in Part II. The main result concerning the strong monomial case will be given by
Theorem \ref{claim:gamma:res} which ensures resolution of singularities in positive characteristic if one could achieve some numerical conditions.
Condition which are achievable for two dimensional schemes.

\end{parrafo}

\end{part}


\begin{part}{Strong monomial case.}\label{sec999}
\setcounter{theorem}{0}
\setcounter{section}{8}

\vspace{0.15cm}

\begin{parrafo}
In this second part we address the proof of \ref{pi14}, (2).
Given a simple differential algebra $\G$ in $V^{(d)}$ and a sequence of transformations, say
\begin{equation}\label{seqsmc}
\xymatrix@R=0pc@C=0pc{
\G & & & & & \G_1 &  & & & &  &  & & & &   \G_r\\
V^{(d)}  &  & & & &   V^{(d)}_{1}\ar[lllll]_{\pi_1}   & & & & & \dots \ar[lllll] &  & & & &   V^{(d)}_{r},\ar[lllll]_{\pi_r}  
}\end{equation}
we have defined:
\begin{enumerate}
\item[$\bullet$] a function $v-ord^{(d-1)}(\G_r)(-):\Sing(\G_r)\longrightarrow \mathbb{Q}$ (see \ref{def_vord}).
\item[$\bullet$] a monomial algebra $\m_r W^s$ in $V^{(d)}_r$, called tight monomial algebra, supported on the exceptional locus of the sequence (see \ref{asert}).
\end{enumerate}

We begin this part by showing that the inequality
\begin{equation}\label{eqdesg}
v-ord^{(d-1)}(\G_r)(x)\geq \ord(\m_r W^s)(x)
\end{equation}
holds at any closed point $x\in\Sing(\G_r)$. The main objective is to study the case in which equality is achieved at any closed point of $\Sing(\G_r)$. This will be called \emph{strong monomial case} in Definition \ref{def_smc}, and we prove that:
\begin{enumerate}
\item The strong monomial case is stable under transformations (Proposition \ref{prop:stab:smc}).
\item It parallels the so called monomial case in characteristic zero. Namely that if $\G_r$ is in the strong monomial case, then a combinatorial resolution leads to a resolution of $\G$ (Theorem \ref{claim:gamma:res}).
\end{enumerate}

\begin{remark}\label{rmktom}
Let $p\P(\beta_r,z,f_{p^e})$ be a $p$-presentation compatible with $\m_r W^s$ (Definition \ref{defmon}) and well-adapted to $\G_r$ at $\x=\beta_r(x)$ for $x\in\Sing(\G_r)$ (Definition \ref{defweladap}). We denote $\x=\beta_r(x)$ along this section. In this case, $z$ must be an element of order one at the local ring $\calo_{V^{(d)},x}$ (see Remark \ref{rmk56}), and $f_{p^e}(z)=z^{p^e}+a_1z^{p^e-1}+\dots+a_{p^e}$, where $a_jW^j\in\m_r W^s$ for $j=1,\dots,p^e$. In addition, $(\R_{\G,\beta})_r\subset \m_r W^s$ (Definition \ref{defmon} (2)).

It follows from the previous discussion that $\ord((\R_{\G,\beta})_r)(\x)\geq\ord(\m_rW^s)(\x)$ and $\frac{\nu_{\x}(a_j)}{j}\geq \ord(\m_r W^s)(\x)$ for $j=1,\dots, p^e$. In particular,
$$v-ord^{(d-1)}(\G_r)(x)=\min\Big\{\frac{\nu_{\x}(a_{p^e})}{p^e},\ord((\R_{\G,\beta})_r)(\x)\Big\}\geq \ord(\m_r W^s)(\x).$$
This proves (\ref{eqdesg}) for any $p\P(\beta_r,z,f_{p^e})$ as above. 
\end{remark}

\end{parrafo}

\begin{parrafo}\label{par:moncas}
We shall say that a Rees algebra is within the \emph{monomial case} when its elimination algebra is monomial, as stated in Theorem \ref{thm:BV}. 
We shall assume here that $\G_r$ is in the monomial case. Namely, that $(\R_{\G,\beta})_r=\mathcal{N}_rW^s$ is a monomial algebra. 
Without lost of generality fix $s\in\mathbb{Z}$ as in (\ref{asert}) so
\begin{equation}\label{2monos}
\mathcal{N}_rW^s=I(H_1)^{\alpha_1}\dots I(H_r)^{\alpha_r}W^s\ \hbox{ and }\ 
\m_r W^s=I(H_1)^{h_1}\dots I(H_r)^{h_r} W^s \ (\hbox{see }\ref{asert}),
\end{equation}
and note that the monomial $\m_r$ divides $\mathcal{N}_r$ (i.e., $\alpha_i\geq h_i$ for any $i=1\dots, r$).
\end{parrafo}

\begin{definition}\label{def_smc} $\G_r$ is said to be within the \emph{strong monomial case at a closed point $x\in \Sing(\G_r)$} if
$$v-ord^{(d-1)}(\G_r)(x)= \ord(\m_r W^s)(x).$$
We say that $\G_r$ is within the \emph{strong monomial case} if this condition holds at any \underline{{\sl closed}} point $x\in \Sing(\G_r)$. 
\end{definition}

The following provides a characterization of  this case.

\begin{theorem}\label{thm:char:stron}{\rm({\bf Characterization of the strong monomial case)}}. 
Fix a closed point $x\in\Sing(\G_r)$. Let  $p\P(\beta_r,z,f_{p^e})$ be well-adapted to $\G_r$ at $\x=\beta_r(x)$ and compatible with the tight monomial algebra $\m_r W^s$. The algebra $\G_r$ is in the strong monomial case at $x$ if and only if one of the following conditions holds in an open neighborhood, either
\begin{itemize}
\item[(i)]   $(\R_{\G,\beta})_r=\m_r W^s$ , or
\item[(ii)]  The $\calo_{V^{(d-1)}}$-algebra spanned by $a_{p^e}W^{p^e}$, namely $\calo_{V^{(d-1)}}[a_{p^e}W^{p^e}]$, has the same integral closure as $\calo_{V^{(d-1)}}[\m_r W^s]$.
\end{itemize}
The first condition holds if and only if $v-ord^{(d-1)}(\G_r)(x)=\ord((\R_{\G,\beta})_r)(\x)$.
\end{theorem}

\begin{proof}
(i) Fix $x\in\Sing(\G_r)$ and denote by $E_x=\{H_{i_1},\dots,H_{i_\ell}\}$ the set of exceptional hypersurfaces containing $x$. Let $\Lambda_x=\{i_1,\dots,i_\ell\}$ be the set of indexes of $E_x$.

If  $v-ord^{(d-1)}(\G_r)(x)=\ord((\R_{\G,\beta})_r)(\x)$ and  $\G_r$ is within the strong monomial case at $x\in\Sing(\G_r)$, then
$$\sum_{i\in\Lambda_x}\alpha_i=\ord((\R_{\G,\beta})_r)(\x)=\ord(\m_r W^s)(\x)=\sum_{i\in\Lambda_x}h_i.$$
Since $\alpha_i\geq h_i$ for any $i$, then $h_i=\alpha_i$ for any $i\in\Lambda_x$. So $\m_r W^s=(\R_{\G,\beta})_r$  at $\x$.

Conversely, if $\m_r W^s=(\R_{\G,\beta})_r$ locally at $\x$, then the inequality $v-ord^{(d-1)}(\G_r)(x)\geq\ord(\m_r W^s)(x)$ in (\ref{eqdesg}) must be an equality since 
$v-ord^{(d-1)}(\G_r)(x)=\min\Big\{\frac{\nu_{\x}(a_{p^e})}{p^e},\ord((\R_{\G,\beta})_r)(\x)\Big\}$, and $\frac{\nu_\x(a_j)}{j}\geq \ord(\m_rW^s)(\x)=\ord(\R_{\G,\beta})(\x)$ (see Remark \ref{rmktom}).

(ii) Suppose now that $v-ord^{(d-1)}(\G_r)(x)=\frac{\nu_\x(a_{p^e})}{p^e}<\ord((\R_{\G,\beta})_r)(\x)$. 
%
%
Recall that $a_{p^e}W^{p^e}\in\m_r W^s$, so $a_{p^e}W^{p^e}\in\m_r^{[p^e]} W^{p^e}$, where $\m_r^{[p^e]}$ is the monomial ideal defined in Remark \ref{rmkintM}. 

Set $a_{p^e}=\m_r^{[p^e]}a'$, for some $a'\in\calo_{V^{(d-1)},\x}$. We claim that $a'$ is a unit, and that $\m_rW^s$ and $\m_r^{[p^e]}W^{p^e}$ have the same integral closure.

Assume first that $\G_r$ is in the strong monomial case at $x\in\Sing(\G_r)$, so 
$\ord(\m_rW^s)(x)=\frac{\nu_\x(a_{p^e})}{p^e}$. Then 
 $\ord(\m_r W^s)(\x)=\frac{\nu_\x(a_{p^e})}{p^e}\geq\ord(\m_r^{[p^e]}W^{p^e})(\x)$. Remark \ref{rmkintM} implies that equality must hold and  both monomial algebras, $\m_r^{[p^e]}W^{p^e}$ and $\m_r W^s$, have the same integral closure. In particular,  $a'$ is a unit, and the algebra spanned by $a_{p^e}W^{p^e}$ has the same integral closure as that of $\m_r W^s$.
 
 Conversely, if the algebras generated by $a_{p^e}W^{p^e}$ and $\m_r W^s$ have the same integral closure,  we claim that $v-ord^{(d-1)}(\G_r)(x)=\frac{\nu_\x(a_{p^e})}{p^e}$. To show this, use (\ref{eqdesg}) to check that
 $\frac{\nu_\x(a_{p^e})}{p^e}\geq v-ord^{(d-1)}(\G)(x)\geq \ord(\m_r W^s)(\x)=\frac{\nu_\x(a_{p^e})}{p^e}.$
 So, finally, $\G_r$ is in the strong monomial case at $x\in\Sing(\G_r)$.
\end{proof}

\begin{remark}\label{rmk:smc}
Let $\G_r$ be in the strong monomial case at the closed point $x\in\Sing(\G_r)$. Then we claim that the following conditions hold for a $p$-presentation $p\P$ in the conditions of Theorem \ref{thm:char:stron}:
\begin{enumerate}
\item In case (i), the transversal parameter $z$ defines a hypersurface of maximal contact. In particular, there exists an open neighborhood of $\x$ where $(\R_{\G,\beta})_r=\mathcal{N}_rW^s=\m_r W^s$.
\item In case (ii), the monomial algebra can be described as $\m_r W^{p^e}$ (i.e., $s=p^e$) where $\m_r$ is not a $p^e$-th power. 
\end{enumerate}

For (1), note that $a_jW^j\in(\R_{\G,\beta})_r$ and that $zW$ fulfills the integral condition 
$$\lambda^{p^e}+(a_1W^1)\lambda^{p^e-1}+\dots+(a_{p^e}-f_{p^e}(z))W^{p^e}=0$$

(2) follows from the fact that $p\P$ is well-adapted to $\G_r$ at $\x$ and compatible with $\m_r W^s$ (see Section \ref{sec555}). That is, $\In_\x(a_{p^e})$ is not a $p^e$-th power and $a_{p^e}W^{p^e}=\m_r W^{p^e}$ ($a_{p^e}=\m_r$).

\end{remark}

\begin{lemma}\label{Lem:vord:C}
Let $\G_r$ be in the strong monomial case, and set  $(\R_{\G,\beta})_r=\mathcal{N}_rW^s$ as in \ref{par:moncas}.
Fix $x,y\in\Sing(\G_r)$ so that $\x=\beta_r(x)$ is closed and $x\in\overline{y}$. Let $p\P(\beta_r,z,f_{p^e})$ be well-adapted at $\x$ and $\beta_r(y)$ and compatible with $\m_rW^s$.
 
\begin{enumerate} 
\item[{\bf(A)}] If  $v-ord^{(d-1)}(\G_r)(y)=\ord((\R_{\G,\beta})_r)(\beta_r(y))$, then
\begin{enumerate}
\item[(A1)] there is a dense open set $U\subset  \overline{y}$ so that 
$v-ord^{(d-1)}(\G_r)(x')=\ord((\R_{\G,\beta})_r)(\beta_r(x'))$ and  $\mathcal{N}_rW^s= \mathcal{M}_rW^s$ at any $x'\in U$.

\item[(A2)] $\ord((\R_{\G,\beta})_r)(\beta_r(y))=\ord(\mathcal{M}_rW^s)(\beta_r(y)).$
\end{enumerate}

\item[{\bf(B)}] If $v-ord^{(d-1)}(\G_r)(y)<\ord((\R_{\G,\beta})_r)(\beta_r(y))$, then at each closed point $x\in \overline{y}\ (\subset \Sing(\G_r))$:
$$v-ord^{(d-1)}(\G_r)(x)=\frac{\nu_\x(a_{p^e})}{p^e}<\ord((\R_{\G,\beta})_r)(\x).$$
\end{enumerate}
\end{lemma}

\begin{proof}
\quad {\bf(A)} Note that (A2) follows from (A1) (since $\mathcal{N}_rW^s= \mathcal{M}_rW^s$ at any $\x\in U$).

Let $E_y=\{H_{j_1},\dots,H_{j_\ell}\}$ denote the set of exceptional hypersurfaces containing $y$ with set of indexes $\Lambda_y=\{j_1,\dots,j_\ell\}$. Set $\Lambda_y^{-}=\{1,\dots,r\}\setminus\Lambda_y$. Consider a closed point $x'$ so that $x'\in \overline{y}\setminus\bigcup_{i\in\Lambda_y^{-}}H_i.$
Fix a $p$-presentation $p\P(\beta_r,z,f_{p^e})$ well-adapted to $\G_r$ both at $\beta_r(x')$ and $\beta_r(y)$. Note that, for such $x'$, and since $(\R_{\G,\beta})_r$ is a monomial algebra supported on the exceptional components
$$\ord((\R_{\G,\beta})_r)(\beta_r(x'))=\ord((\R_{\G,\beta})_r)(\beta_r(y))\leq\frac{\nu_{\beta_r(y)}(a_{p^e})}{p^e}\leq\frac{\nu_{\beta_r(x')}(a_{p^e})}{p^e}.$$ 
So in this case, $v-ord^{(d-1)}(\G_r)(x')=\ord((\R_{\G,\beta})_r)(\beta_r(x'))$ and since $\G_r$ is in the strong monomial case,  $\ord(\m_r W^s)(\beta_r(x'))=v-ord^{(d-1)}(\G_r)(x')=\ord((\R_{\G,\beta})_r)(\beta_r(x'))$. Finally, argue as in the Proof of Theorem \ref{thm:char:stron} to conclude that $\alpha_i=h_i$ for all $i\in\Lambda_y$. Hence $(\R_{\G,\beta})_r=\mathcal{N}_rW^s=\m_r W^s$. In particular,
$$\ord((\R_{\G,\beta})_r)(\beta_r(y))=\ord(\m_r W^s)(\beta_r(y)).$$

{\bf(B)} Fix the closed point  $x\in \overline{y}$ and a $p$-presentation $p\P(\beta,z,f_{p^e})$ well-adapted to $\x=\beta_r(x)$ and $\beta_r(y)$ and compatible with $\m_rW^s$. Assume, on the contrary, that $v-ord^{(d-1)}(\G_r)(x)=\ord((\R_{\G,\beta})_r)(\x)\big(=\ord(\m_r W^s)(\x)\big)$. Remark \ref{rmk:smc} (i) ensures that $(\R_{\G,\beta})_r= \mathcal{M}_rW^s$ in a neighborhood of $\x$. In particular, $\ord((\R_{\G,\beta})_r)(\beta_r(y))=\ord(\mathcal{M}_rW^s)(\beta_r(y)),$
so $$v-ord^{(d-1)}(\G_r)(y)=\frac{\nu_{\beta_r(y)}(a_{p^e})}{p^e}<\ord((\R_{\G,\beta})_r)(\beta_r(y))=  \ord(\mathcal{M}_rW^s)(\beta_r(y)),$$ 
which is in contradiction with the fact that $a_{p^e}W^{p^e}\in \mathcal{M}_rW^s$ locally at $y$ (see Remark \ref{rmktom}).
\end{proof}

\begin{corollary}\label{smcgen}
Let $\G_r$ be within the strong monomial case at  a closed point $x\in\Sing(\G_r)$. Let  $y\in\Sing(\G_r)$ be a point so that $x\in\overline{y}$. Then,
$$v-ord^{(d-1)}(\G_r)(y)=\ord(\m_r W^s)(y).$$
\end{corollary}

\begin{proof}
\begin{enumerate}
\item[(A)] When $v-ord^{(d-1)}(\G_r)(y)=\ord((\R_{\G,\beta})_r)(\beta_r(y))$ the assertion is (A2) in Lemma \ref{Lem:vord:C}.
\item[(B)] Assume that $v-ord^{(d-1)}(\G_r)(y)=\frac{\nu_{\beta_r(y)}(a_{p^e})}{p^e}<\ord((\R_{\G,\beta})_r)(\beta_r(y))$. At the closed point $x\in\overline{y}$, Lemma \ref{Lem:vord:C} (B) ensures that $\ord(\m_r W^s)(\x)=v-ord^{(d-1)}(\G_r)(\x)=\frac{\nu_\x(a_{p^e})}{p^e}<\ord((\R_{\G,\beta})_r)(\x)$. Theorem \ref{thm:char:stron} asserts that, locally at $\x$,  the algebras generated by $a_{p^e}W^{p^e}$ and $\m_r W^s$ have the same integral closure, so 
$\frac{\nu_{\beta_r(y)}(a_{p^e})}{p^e}=\ord(\m_r W^s)(\beta_r(y)).$
\end{enumerate}
\end{proof}

\begin{parrafo}
Assume that $\G_r$ is in the strong monomial case. Corollary \ref{smcgen} says that both  functions $v-ord^{(d-1)}(\G_r)(-) $ and $\ord(\m_r W^s)(-)$ take the same value at any point of $\Sing(\G_r)$. In particular, whenever $\G_r$ is in the strong monomial case, the function $v-ord^{(d-1)}(\G_r)$ is upper-semicontinous. 

\vspace{0.15cm}

We now prove that the strong monomial case is stable under transformations.

\end{parrafo}

\begin{remark}\label{krx1}
 When $\G_r$ is within the strong monomial case, and  $C\subset\Sing(\G_r)$ is a permissible center, then Corollary \ref{smcgen} shows that $\beta_r(C)$ is also a permissible center for the algebra generated by $\m_r W^s$ (i.e., $\beta_r(C)\Sing(\m_rW^s)$). In particular, the transform of the tight monomial algebra can be defined.
\end{remark}

\begin{lemma}\label{smctight}
Assume that $\G_r$ is in the strong monomial case. Let $C\subset\Sing(\G_r)$ be an irreducible permissible center. Let $V_r^{(d)}\overset{\ \pi_C}{\longleftarrow}V_{r+1}^{(d)}$ be the monoidal transformation with center $C$. Denote by $\m_{r+1}' W^s$ the transform of $\m_r W^s$ and by $\m_{r+1} W^s$ the tight monomial algebra of $\G_{r+1}$.
Then, $\m_{r+1}' W^s=\m_{r+1}W^s$.
\end{lemma}

\begin{proof} 

By definition, the tight monomial algebra of the transform, say $\G_r$ , is of the form
$$\m_{r+1}W^s=I(H_1)^{h_1}\dots I(H_r)^{h_r}I(H_{r+1})^{h_{r+1}}W^s,$$
where the $H_j$ are the strict transforms of the previous exceptional hypersurfaces ($j=1,\dots,r$) and $H_{r+1}$ is the new exceptional hypersurface introduced by $\pi_C$. Recall that $h_{r+1}=q_{H_{r+1}}\cdot s$ where $q_{H_{r+1}}=v-ord^{(d-1)}(\G_r)(\xi_{C})-1$, and $\xi_{C}$ denotes the generic point of $C$.
On the other hand,
$$\m_{r+1}'W^s=I(H_1)^{h_1}\dots I(H_r)^{h_r}I(H_{r+1})^{\gamma}W^s,$$
where $\frac{\gamma}{s}=\ord(\m_r W^s)(y)-1$ and $y$ denote the generic point of $\beta_r(C)$.
Corollary \ref{smcgen} asserts that in the strong monomial case $v-ord^{(d-1)}(\G_r)(\xi_C)=\ord(\m_r W^s)(y)$. Thus, $\gamma=q_{H_{r+1}}$, and hence $\m_{r+1}' W^s=\m_{r+1}W^s$.
\end{proof}

\begin{proposition}{\rm ({\bf $\tau=1$-stability of the strong monomial case})}. \label{prop:stab:smc}
Suppose that $\G_r$ is within the strong monomial case. Let $C$ be a permissible center. Assume that  $\tau_{\G_r,x}=1$ at a closed point $x\in C$. Consider the monoidal transformation of center $C$, say $V_r^{(d)}\overset{\pi_C}{\longleftarrow}V^{(d+1)}_{r+1}$. Then, over  an neighborhood of $x\in\Sing(\G_r)$, the transform of $\G_r$, say $\G_{r+1}$, is within the strong monomial case.
\end{proposition}

\begin{proof}
Fix a $p$-presentation $p\P(\beta,z,f_{p^e})$ well-adapted to $\G_r$ both at $\x=\beta_r(x)$ and at $\xi_{\beta(C)}$, and compatible with $\m_r W^s$. Here $\xi_{\beta_r(C)}$ denotes the generic point of $\beta_r(C)$. Set $f_{p^e}(z)=z^{p^e}+a_1z^{p^e-1}+\dots+a_{p^e}$.

Proposition \ref{ppreOK} asserts that $p\P_1$ is well-adapted to $\G_{r+1}$ at $\xi_{H_{r+1}}$, the generic point of $H_{r+1}$ (i.e., the new exceptional hypersurface), and $p\P_1$ is compatible with the tight monomial algebra of $\G_{r+1}$, namely $\m_{r+1}W^s$.



We {\bf claim} that $p\P_1$ is well-adapted to $\G_{r+1}$ at $\x'=\beta_{r+1}(x')$ for any closed point $x'\in\Sing(\G_{r+1})$ mapping to $x$. In particular,  that $v-ord^{(d-1)}(\G_{r+1})(x')=\min\{\frac{\nu_{\x'}(a_{p^e}^{(1)})}{p^e},\ord((\R_{\G,\beta})_{r+1})(\x')\}$, where $a_{p^e}^{(1)}$ denotes the independent term of the strict transform $f_{p^e}^{(1)}(z_1)$.

As $\G_r$ is in the strong monomial case at $x$, then either $\m_r W^s=(\R_{\G,\beta})_r=\mathcal{N}W^s$, or the algebras spanned by $a_{p^e}W^{p^e}$ and $\m_r W^s$ have the same integral closure. Lemma \ref{smctight} asserts that the transform $\m_r W^s$ (the tight monomial algebra of $\G_r$) is $\m_{r+1}W^s$, the tight monomial algebra of $\G_{r+1}$. Hence, at $V_{r+1}^{(d)} $, conditions (i) or (ii) in Theorem \ref{thm:char:stron}  are preserved. Thus, if the claim is true, then $\G_{r+1}$ is also in the strong monomial case at $x'$.

We now address the proof of the previous claim: Fix a closed point $x'\in\Sing(\G_{r+1})\cap H_{r+1}$ mapping to $x$ (i.e., $\pi_C(x')=x)$. Recall that, under the hypothesis $\tau_{\G_r,x}=1$, Remark \ref{lem:z1:sing} asserts that $x'\in\{z_1=0\}$, where $z_1$ denotes the strict transform of $z$. We prove now that $p\P_1$ is well-adapted to $\G_{r+1}$ at $\x'$.

We assume that the tight monomial algebra $\m_r W^s$ is of the form
\begin{equation}\label{eq:monred}
I(H_1)^{h_1}\dots I(H_r)^{h_r}W^s\ \hbox{ with }0< h_i< s.
\end{equation}
In order to achieved this (namely, that all $h_i<s$), it suffices to consider a finite sequence of permissible transformations with centers of codimension 2 at
 $V_{r}^{(d)} $.

We divide the proof in the following cases:

{\bf(1)} Assume that $v-ord^{(d-1)}(\G_r)(x)=\ord((\R_{\G,\beta})_r)(\x)\ \big(\!=\ord(\m_r W^s)(\x)\big)$. Theorem \ref{thm:char:stron} ensures that $(\R_{\G,\beta})_r=\mathcal{N}_rW^s=\m_r W^s$ in a neighborhood of $x$ and, in particular, that $\ord((\R_{\G,\beta})_r)(\xi_C)=\ord(\m_r W^s)(\xi_C)$. Note that $\m_{r+1}W^s=(\R_{\G,\beta})_{r+1}$, and again Theorem \ref{thm:char:stron} says that $\G_{r+1}$ is in the strong monomial case (in particular $p\P_1$, is well-adapted to $\G_{r+1}$ at $\x'$).

{\bf(2)} Suppose now that $v-ord^{(d-1)}(\G_r)(x)=\frac{\nu_\x(a_{p^e})}{p^e}<\ord((\R_{\G,\beta})_r)(\x)$. In this case, Theorem \ref{thm:char:stron} says that the algebras spanned by $a_{p^e}W^{p^e}$ and $\m_r W^s$ have the same integral closure. In particular, we can take $s=p^e$ and assume that $a_{p^e}W^{p^e}=u\m_r W^{p^e}$, where $u$ is a unit (see Remark \ref{rmk:smc}). 

The equality $a_{p^e}W^{p^e}=u\m_r W^{p^e}$ implies that 
\begin{equation}\label{numero}
\frac{\nu_{\xi_{\beta_r(C)}}(a_{p^e})}{p^e}=\ord(\m_r W^s)(\xi_{\beta_r(C)})\ \big(\!\leq\ord((\R_{\G,\beta})_r)(\xi_{\beta_r(C)})\big).
\end{equation}
and, as $p\P$ is well-adapted at $\xi_{\beta_r(C)}$, then $v-ord^{(d-1)}(\G_r)(\xi_C)=\frac{\nu_{\xi_{\beta_r(C)}}(a_{p^e})}{p^e}$, and hence
 $h_{r+1}=\nu_{\xi_{\beta_r(C)}}(a_{p^e})-p^e$ is the exponent of $I(H_{r+1})$ in $\m_{r+1}W^{s}$ ($s=p^e$).

{\bf(2.A)} If $\frac{\nu_{\xi_C}(a_{p^e})}{p^e}\not \in\mathbb{Z}_{>0}$, then $h_{r+1}=\nu_{\xi_{H_{r+1}}}(a_{p^e}^{(1)})=\nu_{\xi_{\beta_r(C)}}(a_{p^e})-p^e(\not\equiv 0\mod p^e)<\ord((\R_{\G,\beta})_r)(\xi_{H_{r+1}})$. Notice here that $\In_{\x'}(a_{p^e}^{(1)})$ cannot be a $p^e$-th power since $h_{r+1}\not\equiv 0 \mod p^e$. Hence, $p\P_1$ is well-adapted to $\G_{r+1}$ at $\x'$ (see Definition \ref{defweladap}).

\vspace{0.2cm}

Let us introduce some notation useful for the proof of the remaining cases: Fix a regular system of parameters in the local regular ring $\calo_{V^{(d)},x}$, say $\{z,x_1,\dots,x_{d-1}\}$,  such that:
\begin{itemize}
\item[(i)] the tight monomial algebra is locally generated by a monomial in $x_1,\dots,x_r$ ($r\leq d-1$), say 
$x_1^{h_1}\dots x_r^{h_r}$, and

\item[(ii)]  the permissible center is 
$I(C)=\langle z, x_1,\dots,x_\ell, y_1,\dots,y_m\rangle,$ where $\ell\leq r$ and $y_j=x_{r+j}$ for $j=1,\dots,m$. 
\end{itemize}

As $a_{p^e}W^{p^e}=u\cdot \m_rW^s$ where $u$ is invertible, then 
$$a_{p^e}=ux_1^{h_1}\cdots x_\ell^{h_\ell}x_{\ell+1}^{h_{\ell+1}}\cdots x_r^{h_r}\ \ \ \ \ \hbox{ with }0<h_i<p^e.$$

{\bf(2.B)} Assume that $\frac{\nu_{\xi_{\beta_r(C)}}(a_{p^e})}{p^e} \in\mathbb{Z}_{>0}$ and that $\ell<r$. In this case, one can check that $\In_{\x'}(a_{p^e}^{(1)})$, which is also monomial, is not a $p^e$-th power. In fact, at each chart
$$a_{p^e}^{(1)}=ux_1^{h_1+\dots+h_{\ell}-p^e}\big(\frac{x_2}{x_1}\big)^{h_2}\dots\big(\frac{x_\ell}{x_1}\big)^{h_\ell}x_{\ell+1}^{h_{\ell+1}}\dots x_r^{h_r}\quad \hbox{in the }U_{x_1}\hbox{-chart, or}$$
$$a_{p^e}^{(1)}=uy_1^{h_1+\dots+h_{\ell}-p^e}\big(\frac{x_1}{y_1}\big)^{h_1}\dots\big(\frac{x_\ell}{y_1}\big)^{h_\ell}x_{\ell+1}^{h_{\ell+1}}\dots x_r^{h_r}\quad \hbox{in the }U_{y_1}\hbox{-chart}$$
and $0<h_{\ell+1}< p^e$ (i.e., $h_{\ell+1}\not\equiv 0\mod p^e$). This ensures that $\In_{\x'}(a_{p^e}^{(1)})$ is not a $p^e$-th power, and hence $p\P_1$ is well-adapted at $\x'$.

\vspace{0.15cm}

{\bf(2.C)} Assume that $\frac{\nu_{\xi_{\beta_r(C)}}(a_{p^e})}{p^e} \in\mathbb{Z}_{>0}$ and $\ell=r$. 

 Note here that $\ell=r\geq 2$, since $\m_r W^{p^e}$ is not a $p^e$-th power and $h_1+\cdots +h_r \equiv 0\mod p^e$. 
 
 We prove now that $p\P_1$ is well-adapted at $\x'$ by considering two cases:
 
{\bf(2.C.1)} Firstly suppose that $\frac{\nu_{\xi_{\beta_r(C)}}(a_{p^e})}{p^e}<\ord((\R_{\G,\beta})_r)(\xi_{\beta_r(C)})$. After a finite number of monoidal transformations over $V^{(d)}_r$ at centers of codimension 2, we can assume that $h_{r+1}=0$. Thus, the independent term, say $a_{p^e}^{(1)}$, is  
$$a_{p^e}^{(1)}=u\big(\frac{x_2}{x_1}\big)^{h_2}\dots\big(\frac{x_r}{x_1}\big)^{h_r}\quad \hbox{in the }U_{x_1}\hbox{-chart, or}$$
$$a_{p^e}^{(1)}=u\big(\frac{x_1}{y_1}\big)^{h_1}\dots\big(\frac{x_r}{y_1}\big)^{h_r}\quad \hbox{in the }U_{y_1}\hbox{-chart}.$$
Both cases are analogous, so it suffices to consider the problem at  the $U_{x_1}$-chart. The difference with the discussion in (2.B) appears when considering a closed exceptional point where $a_{p^e}^{(1)}$ is a unit.
We address now this case. Let 
$$f_{p^e}^{(1)}=z_1^{p^e}+a_1^{(1)}z_1^{p^e-1}+\dots+a_{p^e}^{(1)}$$ be the strict transform of $f_{p^e}$.
The assumption $\frac{\nu_{\xi_{\beta_r(C)}}(a_{p^e})}{p^e}<\ord((\R_{\G,\beta})_r)(\xi_{\beta_r(C)})$ ensures that $\ord((\R_{\G,\beta})_{r+1})(\xi_{H_{r+1}})>0$, and hence that $x_1$ divides $a_j^{(1)}$ for $j=1,\dots,p^e-1$ (see Theorem \ref{rmk:x:restr}).

 We claim that if $x'\in\Sing(\G_1)$, then $x'\in\{x_1=0\}\cap\{\frac{x_j}{x_1}=0\}$ for some $j\in\{2,\dots,r\}$. Let  
 $\overline{f}_{p^e}^{(1)}=z_1^{p^e}+\overline{a}_{p^e}^{(1)}$
 be the restriction pf $f_{p^e}^{(1)}$ to $x_1=0$, where $\overline{a}_{p^e}^{(1)}=\overline{u}\big(\frac{x_2}{x_1}\big)^{h_1}\dots\big(\frac{x_r}{x_1}\big)^{h_r}$, where we identify $\overline{a}_{p^e}^{(1)}$ with an element of $\calo_{\beta_r(C)}\big[\big(\frac{x_2}{x_1}\big),\dots,\big(\frac{x_r}{x_1}\big),\big(\frac{y_1}{x_1}\big),\dots,\big(\frac{y_m}{x_1}\big)\big]$. 
This is a polynomial ring in $r-1+m$ variables. Consider the Taylor expansion of $\overline{a}_{p^e}^{(1)}$ at this ring, say
$$Tay\big(\overline{a}_{p^e}^{(1)}\big)=\!\!\!\!\!\!\!\!\!\!\!\!\sum_{\qquad\alpha\in\mathbb{N}^{r-1+m}}\!\!\!\!\!\!\!\!\Delta^{\alpha}(\overline{a}_{p^e}^{(1)})\ T^\alpha$$ 
The operators $\Delta^{\alpha}$ in this expansion are differential operators in
$\calo_{\beta_r(C)}\big[\big(\frac{x_2}{x_1}\big),\dots,\big(\frac{x_r}{x_1}\big),\big(\frac{y_1}{x_1}\big),\dots,\big(\frac{y_m}{x_1}\big)\big],$
relative to the ring $\calo_{\beta_r(C)}$.


Note here that $\overline{u}\in\calo_{\beta_r(C)}$, so in particular, $\Delta^{\alpha}(\overline{a}_{p^e}^{(1)})=\overline{u}\Delta^{\alpha}\big(\big(\frac{x_2}{x_1}\big)^{h_2}\dots \big(\frac{x_r}{x_1}\big)^{h_r}\big).$

Since it is assumed that $h_j<p^e$, it follows that
$$\Delta^{\alpha_j}(\overline{a}_{p^e}^{(1)})=\overline{u}\big(\frac{x_2}{x_1}\big)^{h_2}\dots \big(\frac{x_{j-1}}{x_1}\big)^{h_{j-1}}\big(\frac{x_{j+1}}{x_1}\big)^{h_{j+1}}\dots \big(\frac{x_r}{x_1}\big)^{h_r},$$
for $\alpha_j=(0,\dots,h_j,\dots,0)\in\mathbb{N}^{r-1+m}$.  Therefore, if $\Delta^{\alpha_j}(\overline{a}_{p^e}^{(1)})(\x')=0$, then $\x'\in\big\{\frac{x_\ell}{x_1}=0\big\}$ for some $\ell$.

So, if $x'\in\Sing(\G_1)\cap H_{r+1}\cap U_{x_1}$, then $ \x'\in\cup_{2\leq j\leq r}\big\{\frac{x_i}{x_1}=0\Big\}$. In this case we can argue as in (2.B) to show that $\In_{\x'}(a_{p^e}^{(1)})$ is not a $p^e$-th power, and hence that $p\P_1$ is well-adapted at $\x'$.

\vspace{0.15cm}

{\bf(2.C.2)} According to (\ref{numero}), the only case left is $\frac{\nu_{\xi_{\beta_r(C)}}(a_{p^e})}{p^e}=\ord((\R_{\G,\beta})_r)(\xi_{\beta_r(C)})$ within the case $\ell=r$. The equality
 $\ord(\m_r W^s)(\xi_{\beta_r(C)})=\ord((\R_{\G,\beta})_r)(\xi_{\beta_r(C)})$, implies that $h_i=\alpha_i$ for $i=1,\dots,r$ (see (\ref{2monos})).

By the assumption in the case (2), $\ord(\m_r W^s)(\x)<\ord(\mathcal{N}_rW^s)(\x)$. Thus,  there must be an exceptional hypersurface, say $H$, so that $\x\in H$,  $H$ is not a component of the support of $\m_r$ (of $V(\m_r)$), and $H$ is a component of $V(\mathcal{N}_rW^s)$. That is, $H\not=H_j$ for $j=1,\dots,r$ and $\ord(\mathcal{N}_rW^s)(\xi_H)>0$.

Consider the monoidal transformation along $C$. We may assume that after a finite number of monoidal transformations at centers of codimension $2$, that the new  exceptional hypersurface, say $H_{r+1}$, is not a component of $V(a_{p^e}^{(1)})$. So $a_{p^e}^{(1)}$ is essentially monomial and admits expressions as those two in (2.C.1), both in $U_{x_i}$-charts or in $U_{y_j}$-charts. In addition, $H_{r+1}$ is not a component of $V(\mathcal{N}_{r+1}W^s)$, where $\mathcal{N}_{r+1}W^s$ is the transform of $\mathcal{N}_rW^s$. On the contrary, the strict transform of $H$ is a component of $V(\mathcal{N}_{r+1}W^s)$ and is not a component of $V(a_{p^e}^{(1)})$.

We argue now as in (2.C.1), considering restrictions to the strict transform of $H$, instead of  restrictions to $H_{r+1}$. The same arguments apply to show that $p\P_1$ is well-adapted at $\x'$.
\end{proof}

We may assume the existence of resolution for simple Rees algebras with $\tau\geq 2$. This reduction is possible by decreasing induction on the invariant $\tau$. The following Theorem shows how to increase the invariant $\tau$, under the assumption that $\G_r$ is in the strong monomial case.

\begin{theorem}\label{claim:gamma:res}
Let $\G_r$ be within the strong monomial case. Then, any combinatorial resolution of $\m_r W^s$ can be lifted to a sequence of  transformations of $\G_r$, say 
\begin{equation}\label{eq:seq:smc}
\xymatrix@R=0pc@C=0pc{
\G_r & & & & & \G_{r+1} &  & & & &  &  & & & &   \G_N\\
V_r^{(d)}  &  & & & &   V^{(d)}_{r+1}\ar[lllll]_{\pi_{r+1}}   & & & & & \dots \ar[lllll]_{\pi_{r+2}} &  & & & &   V^{(d)}_{N}\ar[lllll]_{\pi_N}
}
\end{equation}
and if $x\in\Sing(\G_r)$ is a closed point so that $\tau_{\G_r,x}=1$, then 
$\pi^{-1}(x)\cap\Sing(\G_N)=\emptyset.$
Hence, $\tau_{\G_N,x'}\geq2$ for any $x'\in\Sing(\G_N)$.

\end{theorem}

\begin{proof}
Recall that $\m_r W^s(\subset \calo_{V_r^{(d)}}[W])$ is the pull-back of a monomial algebra, say $\m_r W^s$ again $(\subset\calo_{V_r^{(d-1)}}[W])$. What we mean here is that a combinatorial resolution of $\m_r W^s$ in dimension $d-1$ can be lifted to a permissible sequence in dimension $d$.

Fix a closed point $x\in\Sing(\G_r)$ so that $\tau_{\G_r,x}=1$. Proposition \ref{prop:stab:smc} ensures that after a permissible sequence of transformation as (\ref{eq:seq:smc}), the transform $\G_N$ is in the strong monomial case.
In particular, $v-ord^{(d-1)}(\G_N)(x')=\ord(\m_{N} W^s)(x')$ for any closed point $x'\in\Sing(\G_N)$ mapping to $x$. Moreover, by assumption $\ord(\m_N W^s)(x')<1$. That is, 
$\pi^{-1}(x)\cap\Sing(\G_N)=\emptyset.$
\end{proof}

\end{part}

\appendix
\begin{part}{Proofs of Theorems}\label{partf}
\section{Proof of Main Theorem 1.}\label{sec777}

\begin{parrafo}{\bf  Hironaka's weak equivalence}. There are two natural operations on Rees algebras, both will be crucial for a precise formulation of Hironaka's notion of invariance.  Fix a smooth scheme $V^{(d)}$ and a set, say $E=\{H_1,\dots,H_r\}$, of smooth hypersurfaces so that $\cup H_i$ has only normal crossings. Let $\G=\bigoplus I_nW^n$ be a Rees algebra in $V^{(d)}$. Let now
\begin{equation}\label{eqPB}
V^{(d)}\overset{\pi}{\longleftarrow}U
\end{equation}
be defined either by:
\begin{enumerate}
\item[(A)] An open set $U$ of $V^{(d)}$ in Zariski or \'etale topology.

\item[(B)] The projection of $U=V^{(d)}\times \mathbb{A}^n_k$ on the first coordinate. Here, $\mathbb{A}^n_k$ denotes the $n$-dimensional affine scheme (with $n\in\mathbb{Z}_{\geq1}$).
\end{enumerate}

In both cases, there is a naturally defined pull-back of the Rees algebra $\G$ and of the set $E$. This defines a Rees algebra $\G_U$ and a set $E_U$. Here $E_U$ consists of the pull-backs of the hypersurfaces in $E$. The Rees algebra $\G_U$ is defined as:
\begin{enumerate}
\item[(A)] The restriction to $U$ in case $(A)$, i.e., $\G_U=\bigoplus (I_n)_UW^n$.

\item[(B)] The total transforms of each ideal $I_n$, say $I_n^{*}$, in case $(B)$, i.e., $\G_U=\bigoplus I_n^{*}$.
\end{enumerate}
The pull-back defined by $V^{(d)}\overset{\pi}{\longleftarrow}U$ is denoted by:
$$\xymatrix@R=0cm{
\G & \G_U\\
(V^{(d)},E) & (U,E_U)\ar[l]_{\ \ \ \pi}
}$$
Observe here that $\Sing(\G_U)=\pi^{-1}(\Sing(\G)).$

\begin{definition}
A \emph{local sequence} of a Rees algebra $\G$ and a set $E$ is a sequence
\begin{equation}\label{localseq}
\xymatrix@R=0cm{
\G & \G_1 & & \G_r\\
(V^{(d)},E) & (\widetilde{V}^{(d)}_1,E_1)\ar[l]_{\ \ \pi_1} &\dots \ar[l]_{\ \ \pi_2} & (\widetilde{V}^{(d)}_r\ar[l]_{\ \pi_r},E_r)
}\end{equation}
where each $\xymatrix{\widetilde{V}_i^{(d)} & \widetilde{V}^{(d)}_{i+1}\ar[l]_{\ \ \ \pi_{i+1}}}$ is a pull-back or  a monoidal transformation at a center $C_i\subset\Sing(\G_i)$ with normal crossing with the exceptional hypersurfaces in $E_i$ for $i=0,\dots,r-1$. 

\end{definition}

\begin{definition}
Fix two Rees algebras $\G$ and $\G'$ and a set of hypersurfaces with normal crossings $E$  in the smooth scheme $V^{(d)}$. We say that $\G$ and $\G'$ are \emph{weakly equivalent} if:
\begin{enumerate}
\item[i)] $\Sing(\G)=\Sing(\G')$.
\item[ii)] Any local sequence of $\G$, say (\ref{localseq}), define a local sequence of $\G'$ (and vice versa), and $\Sing(\G_i)=\Sing(\G_i')$ for $i=0,\dots,r$.
\end{enumerate}
\end{definition}

\begin{remark}
Note that if $\G$ and $\G'$ are weakly equivalent, then also their transforms $\G_i$ and $\G_i'$ are weakly equivalent. So the weak equivalence is preserved after any local sequence.
Two algebras with the same integral closure are weakly equivalent.
\end{remark}
\end{parrafo}

\begin{parrafo}{\bf On Main Theorem 1}.

\begin{proposition}\label{propHG}
Fix a Rees algebra $\G$ and a presentation $\P=\P(\beta,z,f_n(z))$. Let $H$ be a smooth irreducible hypersurface in $V^{(d-1)}$. Denote by $y$ the generic point of $H$ and assume that $\P$ is in normal form at $y$.
Then, $H$ is a component of $\beta(\Sing(\G))$ if and only if $Sl_y(p\P)\geq1$.
\end{proposition}
\begin{proof}
See Proposition \ref{prop36} i).
\end{proof}

\begin{theorem}{\rm ({\bf Main Theorem 1})}.
Fix a Rees algebra $\G$. Consider a point $x\in\Sing(\G)$ and a $p$-presentation, say $p\P$, well-adapted at $\beta(x)$. The value $Sl(p\P)(\beta(x))$ is completely determined by the weak equivalence class of $\G$.
\end{theorem}

\begin{proof}
Here we sketch the proof of this Theorem. Further details can be found in \cite{BeV2}. Fix $p\P=p\P(\beta,z,f_{p^e}(z))$ well-adapted to $\G$ at $\x=\beta(x)$. Fix $f_{p^e}(z)=z^{p^e}+a_1z^{p^e-1}+\dots+a_{p^e}$ and set $r_j=\nu_\x(a_j)$ for $j=1,\dots,p^e$ and $\ord(\R_{\G,\beta})(\x)=\frac{\alpha}{s}$. Set $q=Sl(p\P)(\x)$, Theorem \ref{rmk:x:restr} says that $Sl(p\P)(\x)=\min\{\frac{\nu_\x(a_{p^e})}{p^e})(\x),\ord^{(d-1)}(\R_{\G,\beta})(\x)\}$.

Recall that $z$ is an element of order $1$ in $\calo_{V^{(d)},x}$.

Consider $V^{(d)}\times \mathbb{A}^1$, the product of $V^{(d)}$ with the affine line. Locally, in a neighborhood of $(x,0)\in V^{(d)}\times\mathbb{A}^1$, we identify $f_{p^e}(z)$ with its pull-back. Consider, in addition,  the natural projection $\widetilde{\beta}_0=\beta\times id: V^{(d)}\times\mathbb{A}^1\longrightarrow V^{(d-1)}\times\mathbb{A}^1$, mapping $(x,0)$ to $(\x,0)$. Finally, identify $\R_{\G,\beta}$ with its pull-back in $V^{(d-1)}\times \mathbb{A}^1$. This defines a $p$-presentation of the pull-back of $\G$ at $V^{(d)}\times \mathbb{A}^1$, say again $p\P$. Note that
$$Sl(p\P)((\x,0))=Sl(p\P)(\x)$$
and that $In_{(\x,0)}(a_{p^e})$ can be naturally identified with $In_\x(a_{p^e})$.

Fix coordinates $\{z,x_1,\dots,x_{a},t\}$ locally at $(x,0)$, where $\{z,x_1,\dots,x_{a}\}$ is a regular system of parameters at $\calo_{V^{(d)},x}$, and $\{x_1,\dots,x_{a}\}$ is a regular system of parameters at $\calo_{V^{(d-1)},\x}$. Consider the monoidal transformation with center $q_0=(x,0)$ and let $q_1$ be the intersection of the new exceptional hypersurface, say $H_1$, and the strict transform of  $x\times\mathbb{A}^1$.
This monoidal transformation at $q_0$ induces a monoidal transformation, say $V^{(d-1)}\times \mathbb{A}^1\longleftarrow V^{(d)}_1$, at $(\beta(x),0)=\widetilde{\beta}_0(q_0)$. Moreover, one can define a smooth morphism $\widetilde{\beta}_1:V^{(d+1)}_1\longrightarrow V^{(d)}_1$. The exceptional hypersurface $H_1\subset V^{(d+1)}_1$ is the pull-back of the exceptional hypersurface in $V^{(d)}_1$. To simplify notation, we denote both by $H_1$.

The point $q_1$ is the origin of the $U_t$-chart, ($U_t=\Spec(\calo_{V^{(d)},x}[\frac{z}{t},\dots,\frac{x_{a}}{t},t])$). The transform of $p\P$, say $p\P_1=p\P_1(\widetilde{\beta}_1,z_1,f_{p^e}^{(1)})$, is defined by
$$f_{p^e}^{(1)}(z_1)=z_1^{p^e}+t^{r_1-1}a'_1z_1^{p^e-1}+\dots+t^{r_{p^e}-p^e}a'_{p^e},\ \hbox{ and }\ (\R_{\G,\beta})_1$$ 
where $a_j'$ are not divisible by $t$, and $\ord((\R_{\G,\beta})_1)(\xi_{H_1})=\frac{\alpha-s}{s},$ where $\xi_{H_1}$ is the generic point of $H_1\subset V^{(d)}_1$. 

This process can be iterated $N$-times, defining a sequence of monoidal transformations at $q_0, q_1,\dots$, $q_{N-1}$, where each $q_j$ is the intersection of the new exceptional component, say $H_j\ (\subset V^{(d+1)}_j)$, with the strict transform of $x\times\A^1$.
The transform of $p\P$ at the final $U_t$-chart, say $p\P_N=p\P_N(\widetilde{\beta}_N,z_N,f_{p^e}^{(N)})$, is defined by
$$f_{p^e}^{(N)}(z_N)=z_N^{p^e}+t^{N(r_1-1)}a'_1z_N^{p^e-1}+\dots+t^{N(r_{p^e}-p^e)}a'_{p^e},\ \hbox{ and }\ (\R_{\G,\beta})_N,$$ 
with $\ord((\R_{\G,\beta})_N)(\xi_{H_N})=\frac{N(\alpha-s)}{s}$. Here $\xi_{H_N}$ is the generic point of $H_N\subset V^{(d)}_N$.

Fix $N>>0$, it may occur that $\Sing(\G_N)\cap H_{N}$ has codimension $2$ in the $d+1$-dimensional ambient space (i.e., $H_{N}\subset  V^{(d)}_N$ is a component of $\widetilde{\beta}_N(\Sing(\G_N))$). 
It can be proven that this occurs if and only if $(r_j-j)>0$ ($j=1,\dots p^e$) and $(\alpha-s)>0$.
In that case, $\Sing(\G_N)\cap H_{N}$ is defined by $\id{z_N,t}$. In particular, $\Sing(\G_N)\cap H_{N}$ is smooth when it has codimension $2$ in $V^{(d+1)}_N$. If this is the case, we look for further transformations defined with centers of codimension $2$ as we explain bellow.

Firstly consider the monoidal transformation of $V^{(d+1)}_N$ with center $\id{z_N,t}$. Set $z_{N+1}=\frac{z_N}{t}$. At the $U_t$-chart, the transform of $p\P_N$, say $p\P_{N+1}$, is defined by
$$f^{({N+1})}_{p^e}(z_{N+1})=z_{N+1}^{p^e}+t^{N(r_1-1)-1}a'_1z^{p^e-1}+\dots+t^{N(r_p^e-p^e)-p^e}a'_{p^e},\ \hbox{ and }\ (\R_{\G,\beta})_{N+1},$$
with $\ord((\R_{\G,\beta})_{N+1})(\xi_{H_{N+1}})=\frac{N(\alpha-s)-s}{s}$.

Now $\Sing(\G_{N+1})\cap H_{N+1}$ has codimension $2$ in $V^{(d+1)}_{N+1}$ if and only if it is described by $\id{z_{N+1},t}$, which is a smooth center.


Consider now, if possible, $\ell$ monoidal transformations at centers of codimension $2$ of the form $\id{z_{N+i},t}$. It gives rise to a sequence
\begin{equation}\label{seqweq}
\xymatrix@C=3pc@R=0pc{
\G_N & \G_{N+1} & & \G_{N+\ell}\\
V^{(d+1)}_{N} & V^{(d+1)}_{N+1}\ar[l]_{\pi_N} &\dots\ar[l]_{\ \ \  \ \pi_{N+1}} & V^{(d+1)}_{N+\ell}\ar[l]_{\pi_{N+\ell-1}} 
}\end{equation}
Geometrically, this sequence can be interpreted as follows: Set $H_{N+i-1}$ the exceptional hypersurface introduced by $\pi_{N+i-1}$. The center of the $N+i$-th monoidal transformation, say $\pi_{N+i}$, is defined by $H_{N+i}\cap\Sing(\G_{N+i})$, which is assumed to be of codimension $2$. The sequence (\ref{seqweq}) induces a sequence
$$\xymatrix@C=1.5pc@R=0pc{
V^{(d)}_{N} & V^{(d)}_{N+1}\ar[l] &\dots\ar[l] & V^{(d)}_{N+\ell}\ar[l]\\
(\R_{\G,\beta})_N & (\R_{\G,\beta})_{N+1} & & (\R_{\G,\beta})_{N+\ell}
}$$
where each transformation is the blow-up at the exceptional hypersurface $H_{N+i}\subset V^{(d)}_{N+i}$. Hence each transformation is the identity map.

After the $N+\ell$ monoidal transformations, the exponent of $t$ in the $j$-th coefficient of $f_{p^e}^{N+\ell}$ is $N(r_j-j)-\ell j$; and $\ord((\R_{\G,\beta})_{N+\ell})(\xi_{H_{N+\ell}})=\frac{N(\alpha-s)}{s}-\ell$. Therefore, $\id{z_{N+\ell},t}$ is a permissible center if and only if $N(r_j-j)-\ell j\geq j$ ($j=1,\dots,p^e$) and $N(\alpha-s)-\ell s\geq s$. In particular, this requires that
$$\ell\leq\min_{1\leq j\leq n}\Big\{N\big(\frac{r_j}{j}-1\big)-1,\ N\big(\frac{\alpha}{s}-1\big)-1 \Big\}=N\big(Sl(p\P)(\x)-1\big)-1.$$

Set 
$$\ell_{N}=\lceil N(q-1)-1\rceil.$$
 We claim that this is the highest length of a sequence as (\ref{seqweq}). Namely, that $H_{N+\ell_N}$ is not a component of $\widetilde{\beta}_{N+\ell_N}(\Sing(\G_{N+\ell_N}))$ in $V^{(d)}_{N+\ell_N}$. 

Check that $Sl(p\P_{N+\ell_N})(\xi_{H_{N+\ell_N}})<1$, where $\xi_{H_{N+\ell_N}}$ is the generic point of $H_{N+\ell_N}$. We show now that $p\P_{N+\ell_N}$ is well-adapted at $\xi_{H_{N+\ell_N}}$. This, together with Proposition \ref{propHG}, would ensure that the previous claim holds.

Firstly, suppose that $q=Sl(p\P)(\x)=\ord(\R_{\G,\beta})(\x)=\frac{\alpha}{s}$. In this case, $N(\alpha-s)-\ell_N\cdot s\leq N(r_{p^e}-p^e)-\ell_N\cdot p^e$. So, $Sl(p\P_{N+\ell_N})(\xi_{H_{N+\ell_N}})=\ord((\R_{\G,\beta})_{N+\ell_N})(\xi_{H_{N+\ell_N}})$ and hence $p\P_{N+\ell_N}$ is well-adapted to $\xi_{H_{N+\ell_N}}$

Assume now that $q=Sl(p\P)(\x)=\frac{\nu_\x(a_{p^e})}{p^e}=\frac{r_{p^e}}{p^e}<\ord(\R_{\G,\beta})(\x)$ and that $N\frac{r_{p^e}}{p^e}\not\in\mathbb{Z}$. Then, $N(r_{p^e}-p^e)-\ell_N\cdot p^e\leq N(\alpha-s)-\ell_N\cdot s$, so $Sl(p\P_{N+\ell_N})(\xi_{H_{N+\ell_N}})=\frac{\nu_{\xi_{H_{N+\ell_N}}}(a_{p^e})}{p^e}$ and $0<Sl(p\P_{N+\ell_N})(\xi_{H_{N+\ell_N}})<1$. This ensures the claim.

Finally assume that $q=Sl(p\P)(\x)=\frac{\nu_\x(a_{p^e})}{p^e}=\frac{r_{p^e}}{p^e}<\ord(\R_{\G,\beta})(\x)$ and that $N\frac{r_{p^e}}{p^e}\in\mathbb{Z}$. 
Note that $Sl(p\P_{N+\ell})(\xi_{H_{N+\ell}})=\frac{\nu_{\xi_{H_{N+\ell_N}}}(a'_{p^e})}{p^e}=0<\ord^{(d-1)}((\R_{\G,\beta})_{N+\ell_N})(\xi_{H_{\ell+N}})$, and that $In_{\xi_{H_N+\ell_N}}(a'_{p^e})$ can be naturally identified with $In_\x(a_{p^e})$, which is not a $p^e$-th power (as $p\P$ is well-adapted at $\x$). Proposition \ref{propHG}  ensures finally that $H_{N+\ell_N}$ is not a component of $\beta_{N+\ell_N}(\Sing(\G_{N+\ell_N}))$.

The previous arguments show that the rational number $q=Sl(p\P)(\x)$ is completely characterized by the weak equivalence class of $\G$. To this end, note that
$$\lim_{N\to\infty}\frac{\ell_N}{N}=q-1.$$
\end{proof}

Further consequences of the previous discussion are the following:

\begin{corollary}
Let $\G$ be a Rees algebra. Fix a $p$-presentation $p\P=p\P(\beta,z,f_{p^e}(z))$ well-adapted to $\G$ at $x\in\Sing(\G)$. Then, 
$$v-ord^{(d-1)}(\G)(x)=Sl (p\P)(\beta(x)).$$
\end{corollary}

\begin{corollary}
Let $\G$ be a Rees algebra. Fix two transversal projections $V^{(d)}\overset{\beta}{\longrightarrow}V^{(d-1)}$ and $V^{(d)}\overset{\beta'}{\longrightarrow}V'^{(d-1)}$. For any $x\in\Sing(\G)$
$$\beta-\ord(\G)(\beta(x))=\beta'-\ord(\G)(\beta'(x))\ (=v-ord^{(d-1)}(\G)(x)).$$
\end{corollary}
\end{parrafo}

\section{{The tight monomial algebra and Proof of Main Theorem 2.}}\label{sec888}

\begin{parrafo}

We address here the Proof of Main Theorem 2.

\end{parrafo}

\begin{theorem} \label{thm:tightmon}{\rm ({\bf Main Theorem 2}).}
Fix a sequence of permissible transformations as (\ref{unaseq}). Let $\m_r W^s$ denote the tight monomial algebra defined in \ref{asert}.
Then, at any closed point $x\in\Sing(\G_r)$,
$\m_r W^s$ has monomial contact with $\G_r$, i.e., there is a  $\beta_r$-transversal section $z$ of order one at  $\calo_{V^{(d)}_r,x}$ for which
$$\G_r\subset \id{z}W\odot \m_r W^s.$$
\end{theorem}

\begin{proof}

Assume by induction in $r$ that, locally at any closed point $x\in\Sing(\G_r)$
 the algebra $\m_r W^s$ has monomial contact with $\G_r$, i.e., for some $\beta_r$-transversal section $z'$ of order one at the point,
$$\G_r\subset \id{z'}W\odot \m_r W^s.$$
The condition is vacuous  for $r=0$. 

Let $C$ be a permissible center, and consider the monoidal transformation at $C$, say $V_r^{(d)}\overset{\pi_C}{\longleftarrow}V_{r+1}^{(d)}$. The task is to prove that $\G_{r+1}$ has monomial contact with the new tight monomial algebra, say 
\begin{equation}\label{algo}
\m_{r+1}W^s=\calo_{V^{(d-1)}}[I(H_1)^{h_1}\dots I(H_r)^{h_r}I(H_{r+1})^{h_{r+1}}W^s].
\end{equation}

Fix a $p$-presentation $p\P'_r=p\P'_r(\beta_r,z',f'_{p^e})$ involving $z'$ at $V_r^{(d)}$. Note that the induction hypothesis ensures that $p\P'_r$ is compatible with $\m_rW^s$. Proposition \ref{simult:adap} B) applies here to show that $p\P'_r$ can be modified into a new $p$-presentation, say $p\P_r$ (doing a change of variables of the form $z=z'+\alpha$), so that $p\P_r$ is compatible with $\m_rW^s$ and also well-adapted to $\G_r$ both at $x$ and $\xi_{\beta(C)}$, the generic point of $\beta(C)$

%
%

%
%

We claim that, locally at any closed point $x'\in\Sing(\G_{r+1})$ mapping to $x$, there is a $p$-presentation with the properties:
\begin{itemize}
\item it is compatible with the strict transform of  the monomial algebra $\m_r W^s$, 
\item it is well-adapted to $\G_{r+1}$ at $\xi_{H_{r+1}^{(d)}}$. 
\end{itemize}
That is, locally at any closed point $x'\in\Sing(\G_{r+1})$, there is a $p$-presentation which is well-adapted simultaneously to every $\xi_{H_i^{(d-1)}}$ ($i=1,\dots,r+1$). This, in particular, ensures our task.

If $x'\not\in H_{r+1}^{(d)}$, then Remark \ref{rmk62same} shows that there is an identification between the $p$-presentations $p\P_r$ of $\G_r$ and $p\P_{r+1}$ of $\G_{r+1}$ (in an open subset). Thus the claim follows straightforward  in this case. 

Suppose that $x'\in \Sing(\G_{r+1})\cap H_{r+1}^{(d)}$.

Firstly, we address the claim under the assumption that $\In_x(f_{p^e})=Z^{p^e}$. In this case, $\pi^{-1}(x)\cap\Sing(\G_{r+1})\subset\{z_1=0\}$, where $z_1$ denotes the strict transform of $z$ (see Remark \ref{lem:z1:sing}). Moreover, $p\P_{r+1}$ is well-adapted to $\G_{r+1}$ at $\xi_{H_{r+1}^{(d)}}$ (see Proposition \ref{ppreOK}). 
Let
$$f_{p^e}^{(1)}(z_1)=z_1^{p^e}+a_1^{(1)}z_1^{p^e-1}+\dots+a_{p^e}^{(1)}$$
be the strict transform of $f_{p^e}(z)=z^{p^e}+a_1z^{p^e-1}+\dots+a_{p^e}$.
Since $a_iW^i\in\m_r W^s$, then $a_i^{(1)}W^i\in \m'W^s$ for $i=1,\dots, p^e$. Here $\m' W^s$ denotes the strict transform of $\m_r W^s$. On the other hand, $a_i^{(1)}W^i\in I(H^{(d)}_{r+1})^{h_{r+1}}W^s$, since $p\P_{r+1}$ is well-adapted at $\xi_{H^{(d)}_{r+1}}$(recall that $q_{H^{(d)}_{r+1}}=\frac{h_{r+1}}{s}$).
Thus $a_i^{(1)}W^i\in\m_{r+1}W^s$ (the new tight monomial algebra).

The same arguments applies here to show that $\m_{r+1}W^s\subset(\R_{\G,\beta})_{r+1}$. Then, $p\P_{r+1}$ is compatible with $\m' W^s$ and well-adapted to $\G_{r+1}$ at $\xi_{H_{r+1}^{(d-1)}}$. Therefore, $p\P_{r+1}$ is compatible with $\m_{r+1}W^s$. Hence,
$$\G_{r+1}\subset\id{z_1}W\odot\m_{r+1}W^s$$
in case $\In_x(f_{p^e})\not=Z^{p^e}$.

Assume now that $\In_x(f_{p^e})\not=Z^{p^e}$, then two different cases can occur:
 
$\bullet$ Suppose firstly that $\tau_{\G,x}\geq 2$ and $\In_x(f_{p^e})=Z^{p^e}+A_{p^e}$ where $A_{p^e}$ is not a $p^e$-th power and free of the variable $Z$. In this case, $Sl(p\P_r)(\xi_{\beta_r(C)})=1$, and this would ensure that $h_{r+1}=0$ in (\ref{algo}).
Let $x'\in\Sing(\G_{r+1})\cap H^{(d)}_{r+1}$ be a closed point such that $\pi_C(x')=x$. Assume that $\beta_{r+1}(x')\in V(\m'_r)(\subset V^{(d-1)}_{r+1})$, where $\m_r'W^s$ denotes the stric transform of $\m_{r}W^s$ in $V_{r+1}^{(d-1)}$. One can check that $x'\in\{z_1=0\}$ as all coefficients $a_i^{(1)}$ vanish at $\beta_{r+1}(x')$ for $i=1,\dots,p^e$. The same argument used before shows that $p\P_{r+1}$ is compatible with $\m_{r+1} W^s$.

Assume now that $\beta_{r+1}(x')\not\in V(\m'_r)$. Locally at $\beta_{r+1}(x')$, the monomial algebra is $\m_{r+1} W^s$ which has the same integral closures as $\calo_{V_{r+1}^{(d-1)}}[W]$ and there is nothing to prove in this case.

\vspace{0.2cm}

$\bullet$ Finally, suppose that $\In_x(f_{p^e}(z))=Z^{p^e}+A_jZ^{p^e-j}+\dots$ with $A_j\not=0$ and $j<p^e$. In this case, $\ord(\R_{\G,\beta})(\xi_{\beta(C)})=1$ and  hence $h_{r+1}=0$ in (\ref{algo}).
Similar arguments as those used before apply here to show the compatibility of the strict transform with the monomial algebra: whenever the point  $x'\in V(\m_r')$, then $x'\in\{z_1=0\}$. If not, the monomial algebra is locally of the form $\calo_{V^{(d-1)}_{r+1}}[W]$. This concludes the proof.
\end{proof}

\end{part}


\end{document}